\documentclass{amsart}
\usepackage{graphicx} 
\usepackage{appendix}
\usepackage{bbm}

\usepackage{amsfonts,amsmath,amsthm,mathtools,amssymb}
\usepackage{tikz}
\usetikzlibrary{calc,decorations.markings,fpu}
\usetikzlibrary{hobby}
\tikzset{equally spaced dots/.style={postaction={decorate,
      decoration={markings,
        mark=
        between positions 0 and 1 step 0.9999/#1
        with
        {
          \node[circle,fill,inner sep=1pt](c-\pgfkeysvalueof{/pgf/decoration/mark info/sequence number}){};
        }
      }}}}
\tikzset{
    set arrow inside/.code={\pgfqkeys{/tikz/arrow inside}{#1}},
    set arrow inside={end/.initial=>, opt/.initial=},
    /pgf/decoration/Mark/.style={
        mark/.expanded=at position #1 with
        {
            \noexpand\arrow[\pgfkeysvalueof{/tikz/arrow inside/opt}]{\pgfkeysvalueof{/tikz/arrow inside/end}}
        }
    },
    arrow inside/.style 2 args={
        set arrow inside={#1},
        postaction={
            decorate,decoration={
                markings,Mark/.list={#2}
            }
        }
    },
}
\usepackage{todonotes}

\newcommand{\R}{\mathbb{R}}

\newcommand{\cA}{\mathcal{A}}

\newcommand{\cS}{\mathcal{S}}

\newcommand{\dd}{\textnormal{d}}
\newcommand{\DD}{\textnormal{D}}

\newcommand{\ddt}[1]{\frac{\dd #1}{\dd t}}
\newcommand{\ddtau}[1]{\frac{\dd #1}{\dd \tau}}

\DeclareMathOperator{\cycl}{Cycl}

\theoremstyle{plain}
\newtheorem{theorem}{Theorem}
\newtheorem{proposition}{Proposition}
\newtheorem{corollary}{Corollary}

\newtheorem{remark}{Remark}
\newtheorem{example}{Example}
\AtBeginEnvironment{example}{%
  \pushQED{\qed}%
}
\AtEndEnvironment{example}{\popQED\endexample}

\theoremstyle{definition}
\newtheorem{definition}{Definition}

\newcommand{\rev}[1]{\textcolor{black}{#1}}

\title[Ergodicity \& slow relation functions]{Ergodicity in planar slow-fast systems through slow relation functions}

\author{Renato Huzak}
\address{Hasselt University, Campus Diepenbeek, Agoralaan Gebouw D, 3590 Diepenbeek, Belgium}
\email{renato.huzak@uhasselt.be (corresponding author)}
\author{Hildeberto Jard\'on-Kojakhmetov}
\address{University Of Groningen, Dynamical Systems, Geometry \& Mathematical Physics, Nijenborg 9, 9747 AG, Groningen, The Netherlands}
\email{h.jardon.kojakhmetov@rug.nl}
\author{Christian Kuehn}
\address{Technical University of Munich, Department of Mathematics, School of Computation Information and Technology, Boltzmannstraße 3, 85748 Garching bei München.}
\email{ckuehn@ma.tum.de}

\date{}

\begin{document}

\maketitle

\begin{abstract}
In this paper, we study ergodic properties of the slow relation function (or entry-exit function) in planar slow-fast systems. It is well known that zeros of the slow divergence integral associated with canard limit periodic sets give candidates for limit cycles. We present a new approach to detect the zeros of the slow divergence integral by studying the structure of the set of all probability measures invariant under the corresponding slow relation function. Using the slow relation function, we also show how to estimate (in terms of weak convergence) the transformation of families of probability measures that describe initial point distribution of canard orbits during the passage near a slow-fast Hopf point (or a more general turning point). We provide formulas to compute exit densities for given entry densities and the slow relation function. We apply our results to slow-fast Li\'{e}nard equations. 
\end{abstract}


\section{Introduction}

This paper is dedicated to describing the relationship between measure-theoretic properties of the slow relation function, and the dynamic behaviour of $C^\infty$-smooth planar slow-fast systems with a curve of singularities (often called critical curve) consisting of a normally attracting branch, a normally repelling branch and a contact point between them. Essentially, we look at planar slow-fast systems with a parabola-like critical curve as in Fig. \ref{fig:parabola}. In our context, the slow relation function, see Definition \ref{def-1} in Section \ref{subsection-model} (also known as entry-exit relation, or entry-exit function \cite{Benoit,DM-entryexit,de2021canard}), is a map $S:\sigma\to\sigma$ (see a generalisation in Section \ref{subsection-densities}) measuring the balance between contraction and expansion along branches of the critical curve, where the section $\sigma$ contains the contact point. Roughly speaking, the slow relation function assigns to every point $p$ on the attracting branch the point $q$ on the repelling branch such that the slow divergence integral along the slow segment $[p,q]$ is equal to zero (see Fig. \ref{fig:parabola}). The slow divergence integral \cite[Chapter 5]{de2021canard} is the integral of the divergence of the fast subsystem (singular perturbation parameter is zero), computed along the critical curve with respect to the so-called slow time (for more details we refer the reader to Section \ref{preliminaries-slow-fast} and Section \ref{subsection-model}).
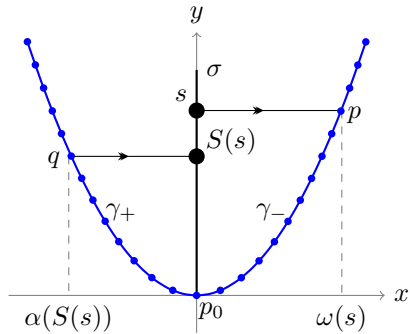
\begin{figure}[htbp]
    \centering
    \begin{tikzpicture}
  \draw[->,thin,gray] (-2.5, 0) -- (2.5, 0) node[right,black] {$x$};
  \draw[->,thin,gray] (0, -.5) -- (0, 3.5) node[above,black] {$y$};
  
  \draw[thick,black] (0,0)--(0,3) node [above,right] {$\sigma$};
  \draw[fill=black] (0,2.46) circle (0.1) node[above,left,yshift=2mm]{$s$};
  \draw[fill=black] (0,1.85) circle (0.1) node[below,right,yshift=2mm]{$S(s)$};
  \draw[black] (0,2.46)--(1.9,2.46) [arrow inside={end=stealth,opt={black,scale=1}}{0.45,0.47}];
  \draw[black] (-1.70,1.85)--(0,1.85) [arrow inside={end=stealth,opt={black,scale=1}}{0.45,0.47}];
  \draw[thin,gray,dashed] (-1.7,1.85)--++(0,-1.85) node[below,black]{$\alpha(S(s))$};
  \draw[thin,gray,dashed] (1.93,2.46)--++(0,-2.46) node[below,black]{$\omega(s)$};
  \draw[scale=0.75, domain=-3:3, smooth, variable=\x, blue,thick,/pgf/fpu/install only={reciprocal},equally spaced dots=26] plot ({\x}, {1/2*\x*\x});
  \node at (-1,1.1) {$\gamma_+$};
  \node at (1,1.1) {$\gamma_-$};
  \node at (0.2,-0.2) {$p_0$};
  \node at (2.1,2.4) {$p$};
   \node at (-1.9,1.8){$q$};
\end{tikzpicture}
    \caption{A slow-fast system with a contact point $p_0$. The blue curve is the curve of singularities (or critical curve) where $\gamma_-$ and $\gamma_+$ represent the attracting and repelling branches, respectively. Under appropriate assumptions (Section \ref{subsection-model}), the slow relation function $S:\sigma\to\sigma$ can be defined, having the following property: the slow divergence integral associated to the critical curve between $\omega(s)$ and $\alpha(S(s))$ is zero. We study measure-theoretic properties of $S$, and relate them with dynamical behaviour of the system.}
    \label{fig:parabola}
\end{figure}
\smallskip

 The slow relation function can be used, for example, to describe (singular) periodic orbits around the contact point, and more generally, to describe transitions across singularities of slow-fast systems \cite{Benoit,DM-entryexit,de2021canard,SDICLE1,haiduc2008horseshoes,yao2022cyclicity}. A natural question that arises for a small but positive value of the singular perturbation parameter is: if an orbit is attracted to the attracting branch near a point $p$, follows that attracting branch, passes near the contact point (called turning point) and follows the repelling branch, how do we detect a point $q$ where the orbit leaves the repelling branch (see Fig. \ref{fig:parabola} and Fig. \ref{fig-motivation-introduction}(a))? We call such orbits canard orbits \cite{de2021canard,kuehn2015multiple,Martin}. Under appropriate assumptions on the slow-fast system, we can find $q$ using the slow relation function (see \cite{Benoit,DM-entryexit} and Proposition \ref{prop:slowrel} in Section \ref{subsection-densities}). The slow relation function (together with the slow divergence integral) also plays an important role in determining the number of limit cycles produced by canard cycles \cite{de2021canard} (i.e.~limit periodic sets consisting of a fast orbit and the portion of the critical curve between the $\alpha$ and $\omega$ limits of that fast orbit, see Fig. \ref{fig-motivation-introduction}(b)). The study of planar canard cycles is motivated by the famous Hilbert’s 16th problem \cite{smale} (see \cite{Artes,Gavrilov,Mardesic,SDICLE1,DPR,1996,PWS2023} and references therein) and by applications (predator-prey models \cite{Broer,Zhu}, electrical circuits, (bio)chemical reactions \cite{Kosiuk,oxid}, neuroscience \cite{hasan2018saddle,moehlis2006canards,wechselberger2013canard,de2015neural}, among many others).  The slow relation function is indeed closely related to the concept of delayed loss of stability \cite{AiSadhu,Zha}, and is also important in fractal analysis of planar slow-fast systems \cite{BoxNovo,BoxRenato,BoxDarko}.
 \smallskip

 \begin{figure}[htb]
	\begin{center}
		\includegraphics[width=9.6cm,height=2.8cm]{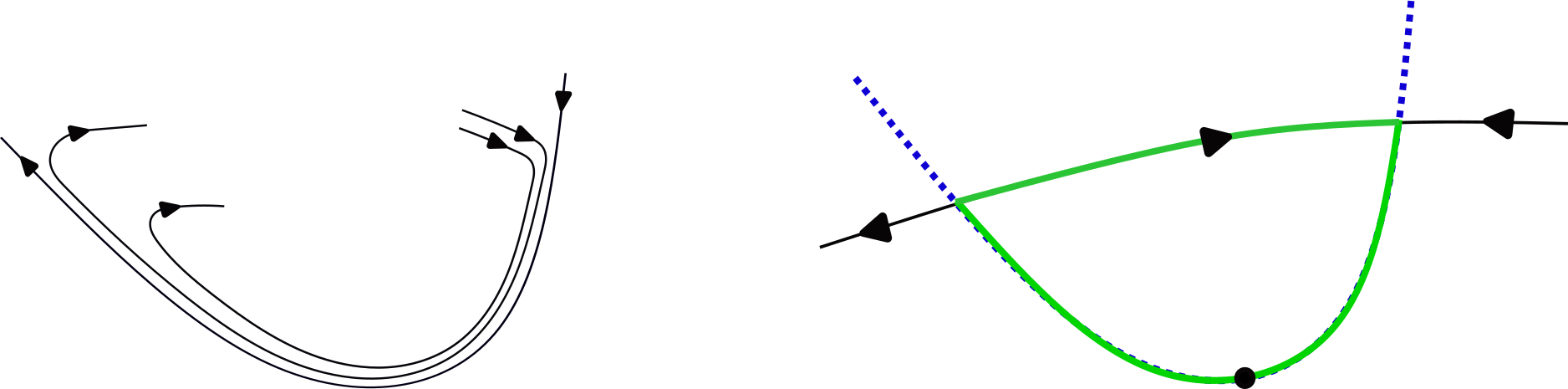}
  {\footnotesize
\put(-220,-13){$(a)$}
\put(-70,-13){$(b)$}
\put(-82,50){$\Gamma$}
  }
		         \end{center} 
	\caption{(a) Canard orbits. (b) Canard cycle $\Gamma$ (green).}
	\label{fig-motivation-introduction}
\end{figure}

One of our main motivations to bring ergodic theory into play, is to be able to describe the behaviour of ensembles of orbits, instead of single ones. For example, \cite{kuehn2017uncertainty} studies the problem of how densities of (uncertain \rev{or random}) initial conditions are transformed, via the flow of the slow-fast system, as the corresponding orbits cross a Hopf bifurcation. In particular,  \cite{kuehn2017uncertainty} finds concrete systems for which, given a density of initial conditions, such a density is transformed in particular ways, or even into a desired one. \rev{We point out that \cite{kuehn2017uncertainty} considers mostly problems at the level $\epsilon=0$ and that weak convergence of exit densities are not discussed. In this paper, we put emphasis precisely on the weak convergence and asymptotics of exit densities, see Section \ref{subsection-densities} and \ref{section-numerics} for more details. Other works that include randomness in the vector field have also considered ``entrance-exit'' asymptotics in the framework of heteroclinic networks, see \cite{bakhtin2011,bakhtin2022}  and references therein. In our context, adding generic stochastic forcing to slow-fast planar vector fields is going to destroy all canard phenomena~\cite{BerglundGentz,BerglundGentzKuehn} involving a long delay near unstable branches, unless such randomness is exponentially small~\cite{Sowers}. 
}
\smallskip

Important connections between ergodicity and slow-fast systems can be found in \cite{EGK,Melbourne} (homogenization of slow-fast systems), \cite{Kuske, Zeghl} (multiscale stochastic ordinary differential equations and bifurcation delay), and \cite{KuLu,KuehnQuenched} (randomness in parameters and bifurcations). See also \cite{Prohens,CollGasullProhens} for results on limit cycles in random planar vector fields.\smallskip

In this paper we deal with smooth nilpotent contact points of arbitrary even contact order (infinite contact order is possible) and odd singularity order. There is an additional assumption: such contact points have finite slow divergence integral. Then we can define the slow relation function. For more details see Section \ref{subsection-model}. The contact order of a slow-fast Hopf point (often called generic turning point) is $2$ and its singularity order is $1$. Non-generic turning points have contact order $2n$ and singularity order $2n-1$ with $n>1$.

The results we present can be classified into two types:
\begin{enumerate}
    \item First, we relate invariant probability measures of the slow relation function with zeros of the slow divergence integral (Theorem \ref{theorem-1} in Section \ref{subsection-limit cycles}). More precisely, we show that the slow divergence integral has no zeros if and only if the slow relation function is uniquely ergodic (see the slow-fast van der Pol system in Example \ref{example-vanderpol}). Furthermore, the slow divergence integral has $k$ zeros (counted without their multiplicity) if and only if the invariant measures are supported on a set with $k+1$ elements (they are convex combinations of $k+1$ Dirac delta measures).  For slow-fast systems with a slow-fast Hopf point or a non-generic turning point, we relate invariant measures of the slow relation function with the cyclicity of canard cycles (Theorem \ref{theorem-2} and Theorem \ref{theorem-3} in Section \ref{subsection-limit cycles}).
    \item  The second type of results is related to entry-exit probability measures. That is, we consider entry measures compactly supported near the attracting branch of the critical curve, and study how they are transformed near the repelling branch, after passage close to a slow-fast Hopf point or a non-generic turning point (Theorem \ref{mainthm-weak} in Section \ref{subsection-densities}). The transformed measures are push-forward measures of the entry measures and we call them the exit measures. The entry and exit measures depend on the singular perturbation parameter denoted by $\epsilon>0$.
    
    Depending on the setup, see more details in Section \ref{sec:statements}, there are two important regions for the dynamics: the tunnel and the funnel regions. In the tunnel region, we show that, if the entry measures converge weakly to a measure $\mu_0$ as $\epsilon\to 0$, then the exit measures converge weakly to the push-forward of $\mu_0$ under the slow relation function, as $\epsilon\to 0$ (Theorem \ref{mainthm-weak}(a)). 
    
    In the presence of both tunnel and funnel regions, separated by a buffer point, the exit measures converge weakly to a more complex measure having two components, one coming from the tunnel behavior (the push-forward of $\mu_0$ under the slow relation function) and the other coming from the funnel behavior (Dirac delta measure concentrated on the image of the buffer point under the slow relation function). Here we also assume that the entry measures converge weakly to a measure $\mu_0$ as $\epsilon\to 0$. For a precise statement of this result we refer the reader to Theorem \ref{mainthm-weak}(b).

    Suppose that $\mu_0$ has density. Then we provide a formula to compute the density of the push-forward of $\mu_0$ under the slow relation function, called the exit density (see Proposition \ref{prop-Frob} in Section \ref{subsection-densities}). 
\end{enumerate}

We often give examples using slow-fast Li\'{e}nard equations (see system \eqref{eq:sf2} in Section \ref{preliminaries-slow-fast}). The main advantage of the Li\'{e}nard model is a simpler expression for the slow divergence integral, see \eqref{SDI-added} in Section \ref{preliminaries-slow-fast}. For example, the divergence of \eqref{eq:sf2} is independent of $y$. We refer to e.g. \cite{SDICLE1,DPR}. Using Proposition \ref{prop-Frob}, we find concrete formulas to compute the exit densities for slow-fast Li\'{e}nard equations (see Corollary \ref{corollary-density} in Section \ref{subsection-densities}). 

\smallskip

For the sake of readability we have chosen to state Theorem \ref{theorem-3} and Theorem \ref{mainthm-weak} for a class of slow-fast Li\'{e}nard equations. However, we point out that they can be stated and proved in a more general framework \cite{DM-entryexit}, even for more degenerate contact points than the nilpotent contact points. In fact, Proposition \ref{prop:slowrel} that we use in the proof of Theorem \ref{mainthm-weak} (Section \ref{proof-thmweak}) is true for a broader class of planar slow-fast systems studied in \cite{DM-entryexit}.
\smallskip

 \smallskip

 The paper is organized as follows. In Section \ref{sect-erg+SF} we recall some basic concepts in ergodic theory and planar slow-fast systems. In Section \ref{sec:statements} we define our planar slow-fast model (see also Section \ref{preliminaries-slow-fast}) and state our main results. Section \ref{section-numerics} is devoted to numerical examples, and in Sections \ref{section-proofs} and \ref{proof-thmweak} we prove the main results.

\section{Preliminaries and some notation}\label{sect-erg+SF}
In Section \ref{sec-ergodic} we recall some important definitions and results in ergodic theory that we will use in our paper. The reader may be referred to, e.g. \cite{Athreya,Bill,katok,Lasota,Sarapa,viana2016foundations} and references therein for further details.
In Section \ref{preliminaries-slow-fast} we recall the notions of curve of singularities, fast foliation, normally hyperbolic singularity, contact point, slow vector field, slow divergence integral, etc., in planar slow-fast systems (for more details see \cite[Chapters 1--5]{de2021canard} and \cite{kuehn2015multiple,Martin}).   

\subsection{Ergodic theory}\label{sec-ergodic}
Assume that $X$ is a measure space. More precisely, $X$ is the short-hand notation for the triplet $(X,\cA,\mu)$ where $(X,\cA)$ is a measurable space with $\cA$ a $\sigma$-algebra of subsets of $X$, for which a measure $\mu:\cA\to [0,+\infty]$ is defined. If $\mu(X)=1$, one usually says that $\mu$ is a probability measure, and calls $(X,\cA,\mu)$ a probability space. In this paper we deal with probability measures. We say that $\mu$ is supported on $A\in\cA$ if $\mu(X\setminus A) = 0$. A map $f:X\to X$ being measurable means that if $A\in\cA$ then $f^{-1} (A) \in\cA$. One further says that $\mu$ is $f$-invariant if $\mu(f^{-1}(A))=\mu(A)$ for all $A\in\cA$. In this case, one can also say that $f$ preserves $\mu$. For example, a Dirac measure $\delta_x$ at $x\in X$, defined by $\delta_x(A)\coloneqq\begin{cases}
    1, & x\in A\\
    0, & x\notin A
\end{cases}$, is $f$-invariant if and only if $x$ is a fixed point of $f$.
\smallskip

An $f$-invariant probability measure $\mu$ is said to be ergodic (w.r.t.~$f$) if for any measurable set $A\in\cA$ such that $f^{-1}(A)=A$ either $\mu(A)=0$ or $\mu(A)=1$. Further, we say that a measurable map $f : X\to X$ is uniquely ergodic if it admits
exactly one invariant probability measure (this invariant probability measure  has to be ergodic w.r.t.~$f$). It is well-known that the space of all $f$-invariant probability measures is convex: if $\mu$ and $\tilde\mu$ are $f$-invariant probability measures, then $(1-t)\mu +t \tilde \mu$, for any $t \in ]0,1[$, is also $f$-invariant. The ergodic probability measures are the extremal
points of this convex set (for more details see e.g. \cite[Proposition 4.3.2]{viana2016foundations}). 
\smallskip

An important question is whether an invariant probability measure exists for a given $f : X\to X$. This leads to the following fundamental result, due to Krylov-Bogolubov \cite[Theorem 4.1.1]{katok}: If $X$ is a compact metric space and $f:X\to X$ a continuous map, then $f$ has an invariant Borel probability measure. Here, $\cA$ is the Borel $\sigma$-algebra of $X$, often denoted by $\mathcal{B}$ (i.e. the $\sigma$-algebra generated by the open (and therefore also closed) subsets of $X$). A probability measure defined on the Borel $\sigma$-algebra $\mathcal{B}$ of a metric (or topological) space $X$ is called a Borel probability measure. In Section \ref{subsection-limit cycles} $X$ will be a compact metric space (a segment in $\mathbb R$) and $f:X\to X$ continuous (see Remark \ref{remark-borel-sigma} in Section \ref{subsection-model}).
\smallskip

One of the most important results in ergodic theory is the Poincar\'{e} recurrence theorem (see Theorem \ref{thm-Poin} in Section \ref{section-proof-1}). Roughly speaking, this result states that $f$-invariant Borel probability measures on a topological space $X$ imply recurrence for $f$ (the definition of recurrent points is given in Section \ref{section-proof-1}). We use the Poincar\'{e} recurrence theorem in the proof of Theorem \ref{theorem-1} (Section \ref{section-proof-1}).

\smallskip


\smallskip

Let $\mu_\epsilon$, with $\epsilon\in ]0,\epsilon_0]$, $\epsilon_0>0$, and $\mu_0$ be Borel probability measures on $\mathbb R$ with the usual Borel $\sigma$-algebra $\mathcal B$. We say that $\mu_{\epsilon}$ converges  weakly (or in distribution) to $\mu_0$ as $\epsilon\to 0$ if 
$$\lim_{\epsilon\to 0}\int_{\mathbb R}\chi(x)\mu_{\epsilon}(dx)=\int_{\mathbb R}\chi(x)\mu_{0}(dx),$$
for every bounded, continuous function $\chi:\mathbb R\to\mathbb R$ (see e.g. \cite{Bill}). The integrals are Lebesgue integrals. 


Let $\mathcal B$ be the usual Borel $\sigma$-algebra of $\mathbb R$ and let $\mu$ be a Borel probability measure on $\mathbb R$. If $f: \mathbb R\to \mathbb R$ is a measurable function, then the push-forward probability measure of $\mu$ is defined as 
$$\mu f^{-1}(A):=\mu\left(f^{-1}(A)\right), \ A\in\mathcal B. $$
Weak convergence is preserved by continuous mappings (see \cite[pg. 20]{Bill}): if $f$ is continuous and $\mu_{\epsilon}$ converges  weakly to $\mu_0$ as $\epsilon\to 0$, then $\mu_{\epsilon}f^{-1}$ converges  weakly to $\mu_0f^{-1}$ as $\epsilon\to 0$.


We will sometimes work with absolutely continuous probability measures w.r.t. the Lebesgue measure on $\mathbb R$ (Section \ref{subsection-densities} and Section \ref{section-numerics}). A Borel probability measure $\mu$ is absolutely continuous w.r.t. the Lebesgue measure if
$$\mu(A)=\int_A D(x)dx, \ A\in\mathcal B,$$
where $D\in L^1(\mathbb R)$ and $D\ge 0$ ($L^1(\mathbb R)$ is the space consisting of all possible Lebesgue integrable functions $\mathbb R\to  \mathbb R$). We call $D$ the density of $\mu$. We refer to \cite[Definition 3.1.4]{Lasota}.

\subsection{Planar slow-fast systems}\label{preliminaries-slow-fast}

We consider a smooth planar slow-fast system defined on an open set $M\subset\mathbb R^2$
\begin{equation}\label{model}
    X_{\lambda,\epsilon}= X_{\lambda,0}+\epsilon Q_\lambda+O(\epsilon^2)
\end{equation}
where $0\leq\epsilon\ll1$ is the singular perturbation parameter, $\lambda$ is a regular parameter kept in a small  neighborhood of $\lambda_0\in\mathbb R^{r}$ (we often write $\lambda\sim\lambda_0$), and $X_{\lambda,0}$ and $Q_\lambda$ are smooth $\lambda$-families of vector fields. In this paper smooth means $C^\infty$-smooth.  We assume that the fast subsystem $X_{\lambda,0}$ has a set of non-isolated singularities $\mathcal S_\lambda$, for all $\lambda\sim\lambda_0$, and that for each $p\in \mathcal S_{\lambda_0}$ there exists an open neighborhood $U\subset M$ of $p$ such that $X_{\lambda,0}=F_\lambda Z_\lambda$ on $U$. Here, $F_\lambda$ is a smooth family of functions with $\nabla F_\lambda(p)\ne 0$, for all $p\in\{F_\lambda=0\}$, and $Z_\lambda$ is a smooth family of vector fields without singularities. It is clear that $\mathcal S_\lambda\cap U=\{F_\lambda=0\}$ and $\mathcal S_\lambda$ is a one-dimensional submanifold of $M$. We call $\mathcal S_\lambda$ the curve of singularities or critical curve.
In \cite[Section 1.1]{de2021canard} $\{U,Z_\lambda,F_\lambda\}$ is called an admissible expression for $X_{\lambda,0}$ near $p$. Notice that the pair $(Z_\lambda,F_\lambda)$ is not unique: we can take $(\rho_\lambda Z_\lambda,\frac{1}{\rho_\lambda} F_\lambda)$ where $\rho_\lambda$ is a nowhere zero smooth function. We denote by $t$ the time variable related to \eqref{model} and call it the fast time.  

\begin{example}\label{example-1}
    A standard example of a planar slow-fast system is the singularly perturbed Li\'{e}nard equation
\begin{equation}\label{eq:sf1}
    \begin{split}
        \epsilon\ddtau{x} &=y- f_\lambda(x)\\
        \ddtau{y} &= g(x,\lambda,\epsilon),
    \end{split}
\end{equation}
where $f_\lambda,g$ are smooth, $(x,y)\in\R^2$, $\lambda\sim\lambda_0\in\R^r$ are parameters, and $0\leq\epsilon\ll1$ is a small parameter accounting for the timescale difference between the fast variable $x$ and the slow variable $y$. $\tau$ is called the slow time variable. The time rescaling $\dd\tau=\epsilon\dd t$ ($t$ is the fast time) leads to the equivalent representation
\begin{equation}\label{eq:sf2}
    Y_{\lambda, \epsilon}:\begin{cases}
        \ddt{x}=y- f_\lambda(x)\\
        \ddt{y} = \epsilon g(x,\lambda,\epsilon),
    \end{cases}
\end{equation}
in which case, for example, $F_\lambda(x,y)=y- f_\lambda(x)$, \rev{$Z_\lambda = \begin{bmatrix}
           1\\
           0 
         \end{bmatrix}$ and $Q_\lambda = \begin{bmatrix}
           0\\
           g(x,\lambda,0) 
         \end{bmatrix}$.} 
The curve of singularities is defined as the set
\begin{equation}
    \cS_\lambda=\left\{ (x,y)\in\R^2\,|\, y=f_\lambda(x) \right\},
\end{equation}
and represents the phase-space and the set of singularities of the limit $\epsilon\to0$ of \eqref{eq:sf1} and \eqref{eq:sf2}, respectively. System $Y_{\lambda, \epsilon}$ is of type \eqref{model}.
\end{example}

The fast foliation of $X_{\lambda,0}$ is denoted by $\mathcal F_\lambda$ and is defined as follows: $\mathcal F_\lambda$ is a smooth $1$-dimensional foliation on $M$ tangent to $Z_\lambda$ in each admissible local expression $\{U,Z_\lambda,F_\lambda\}$ for $X_{\lambda,0}$.
The orbits of the fast flow of $X_{\lambda,0}$, away from $\cS_\lambda$, are located inside the leaves of the fast foliation (we denote by $l_{\lambda,p}$ the leaf
through $p\in M$). For more details we refer to \cite[Chapter 1]{de2021canard}. In Example \ref{example-1}, the fast foliation is given by horizontal lines (see e.g. Fig. \ref{fig:parabola}).
\smallskip

A point $p\in\cS_\lambda$ is called \emph{normally hyperbolic} if the Jacobian matrix
$\DD X_{\lambda,0}(p)$ has a non-zero eigenvalue denoted $E_\lambda(p)$ ($p$ is attracting if $E_\lambda(p)<0$ or repelling if $E_\lambda(p)>0$). Notice that there is one zero eigenvalue with eigenspace $T_p\cS_\lambda$.  The  eigenspace of the nonzero eigenvalue $E_\lambda(p)$ is $T_pl_{\lambda,p}$ and $E_\lambda(p)$ is equal to the trace of $\DD X_{\lambda,0}(p)$ or the divergence of the vector field $X_{\lambda,0}$ w.r.t.~the standard area form on $\mathbb R^2$, computed in $p$. A point $p\in\cS_\lambda$ is called a contact point (between $\cS_\lambda$ and $\mathcal F_\lambda$) when $\DD X_{\lambda,0}(p)$ has two zero eigenvalues. Contact points are nilpotent due to the above-mentioned assumption on $Z_\lambda$ and $F_\lambda$. A curve $\gamma\subset\cS_\lambda$ is called normally attracting (resp. repelling) if every point $p\in\gamma$ is normally hyperbolic and attracting (resp. repelling). For the Li\'{e}nard system \eqref{eq:sf2}, we have $E_\lambda(p)=-f_\lambda'(x)$ with $p=(x,f_\lambda(x))$, and $p$ is normally attracting (resp.~ repelling and contact point) if $f_\lambda'(x)>0$ (resp. $f_\lambda'(x)<0$ and $f_\lambda'(x)=0$).
\smallskip

It is important to define the notion of contact order and singularity order of a contact point $p_0$ for $\lambda=\lambda_0$ (\cite[Section 2.2]{de2021canard}): we call \rev{intersection multiplicity}\footnote{\rev{If the curves are graphs of smooth functions $y=f_1(x)$ and $y=f_2(x)$ in a neighborhood of $p_0$ corresponding to $(x,y)=(0,0)$, then the \emph{intersection multiplicity} is the multiplicity of the zero $x=0$ of $f_1-f_2$.}} at $p_0$ between $\cS_{\lambda_0}$ and the leaf $l_{\lambda_0,p_0}$ the contact order of $p_0$ and denote it by $\mathbbmss n$. Moreover, for any admissible expression  $\{U,Z_\lambda,F_\lambda\}$ for $X_{0,\lambda}$ near $p_0$ and for any area form $\Omega$ on $U$, the order at $p_0$ of the function $\Omega(Q_{\lambda_0},Z_{\lambda_0})|_{\cS_{\lambda_0}\cap U}:  p\in \cS_{\lambda_0}\cap U\mapsto \Omega(Q_{\lambda_0},Z_{\lambda_0})( p)$ is called the singularity order of $p_0$, denoted by $\mathbbmss m$. The definition of singularity order is independent of the choice of the admissible expression near $p_0$ and $\Omega$ (see \cite[Lemma 2.1]{de2021canard}). For \eqref{eq:sf2} in Example \ref{example-1} with a contact point $p_0=(0,0)$, $\mathbbmss n\ge 2$ is equal to the order at $x = 0$ of $f_{\lambda_0}(x)$ \rev{(i.e. the multiplicity of zero $x=0$ of $f_{\lambda_0}(x)$)} and $\mathbbmss m\ge 0 $ is the order at $x = 0$ of $g(x,\lambda_0,0)$ (see also Remark \ref{remark-even} in Section \ref{subsection-model}).

\smallskip

Let $p\in\cS_\lambda$ be normally hyperbolic. Let $ \hat Q_\lambda(p)\in T_p\cS_\lambda$ be the linear projection of $Q_\lambda(p)$ on $T_p\cS_\lambda$ in the direction parallel to the eigenspace $T_pl_{\lambda,p}$ defined above (recall that the vector field $Q_\lambda$ comes from \eqref{model}). The family $\hat Q_\lambda$ is called the slow vector field, and its flow is called the slow dynamics. The time variable of the slow dynamics is the slow time $\tau=\epsilon t$. This definition and the classical one using center manifolds are equivalent (for more details see \cite[Chapter 3]{de2021canard}). If we take \eqref{eq:sf2}, then we get 
 \rev{
\begin{equation}\label{eq-SVF-example}
    \hat Q_{\lambda}:\begin{cases}
        \frac{dx}{d\tau}=\frac{g(x,\lambda,0)}{f_\lambda'(x)}\\
        \frac{dy} {d\tau}= g(x,\lambda,0),
    \end{cases}
\end{equation}
}
when $f_\lambda'(x)\ne 0$.

\smallskip

Let $\gamma\subset \cS_\lambda$ be a normally hyperbolic segment not containing singularities of the slow vector field $\hat Q_\lambda$. We define the slow divergence integral \cite[Chapter 5]{de2021canard} associated to $\gamma$ as
\begin{equation}
    \label{SDI-classicaldef}
 I(\gamma,\lambda)=\int_{ \tau_1}^{ \tau_2}E_{\lambda}(\tilde\gamma( \tau))\dd \tau,
\end{equation}
where $E_\lambda$ is the non-zero eigenvalue function defined above, $\tilde\gamma:[ \tau_1, \tau_2]\to\mathbb R^2$, $\tilde \gamma'( \tau)=\hat Q_\lambda(\tilde\gamma( \tau))$ and $\tilde\gamma ( \tau_1)$ and $\tilde\gamma( \tau_2)$ are the end points of the segment $\gamma$. The segment $\gamma$ is parameterized by the slow time $\tau$. This definition does not depend on the choice of parameterization $\tilde \gamma$ of $\gamma$. Note that $I(\gamma,\lambda)$ is the integral of the divergence of the fast subsystem $X_{\lambda,0}$ computed along $\gamma$ w.r.t. the slow time $\tau$. If $\gamma$ is normally attracting (resp. repelling), then $I(\gamma,\lambda)$ is negative (resp. positive). We point out that the slow divergence integral is invariant under smooth equivalences\footnote{Smooth equivalence means smooth coordinate change and division by a smooth positive function.}, see \cite[Section 5.3]{de2021canard} and Section \ref{subsection-model}. 

Consider \eqref{eq:sf2}. Let $\gamma\subset \cS_\lambda$ be a normally hyperbolic segment parameterized by $x\in [x_1,x_2]$, $x_1<x_2$. Assume that the slow vector field \eqref{eq-SVF-example} has no singularities in $\gamma$ and points, for example, from $x_2$ to $x_1$. Then
\begin{equation}\label{SDI-added}I(\gamma,\lambda)=-\int_{x_2}^{x_1}\frac{\left(f_{\lambda}'(x)\right)^2}{g(x,\lambda,0)}\dd x. 
\end{equation}
Note that the divergence is given by $-f_\lambda'(x)$ and $\dd\tau=\frac{f_\lambda'(x)}{g(x,\lambda,0)}\dd x$, using the $x$ component of \eqref{eq-SVF-example}.

Based on \cite[Definition 5.2]{de2021canard}, in Section \ref{subsection-model} we generalise the definition \eqref{SDI-classicaldef} of the slow divergence integral. We allow the presence of a contact point in one of the boundary points of the segment $\gamma$. This plays an important role when we introduce the notion of slow relation function (see Definition \ref{def-1} in Section \ref{subsection-model}).

\section{Assumptions and statement of the results}\label{sec:statements}
In Section \ref{subsection-model} we focus on the slow-fast family $X_{\lambda,\epsilon}$ defined in \eqref{model} and make some assumptions on $\mathcal S_\lambda$, $\mathbbmss m$, $\mathbbmss n$ and $\hat Q_{\lambda}$. Then we define the slow relation function. We state our main results in Section \ref{subsection-limit cycles} (Theorem \ref{theorem-1}--Theorem \ref{theorem-3}) and Section \ref{subsection-densities} (Theorem \ref{mainthm-weak}). See also Proposition \ref{prop-Frob} and Corollary \ref{corollary-density} in Section \ref{subsection-densities}.

\subsection{Assumptions and slow relation function}\label{subsection-model}
Consider system $X_{\lambda,\epsilon}$. We use the notation from Section \ref{preliminaries-slow-fast}. First we assume that the curve of singularities $\cS_{\lambda_0}$ consists of a normally attracting branch, a normally repelling branch and a contact point between them.\\
\\
\textbf{Assumption 1} We have $\cS_{\lambda_0}=\gamma_-\cup\{p_0\}\cup \gamma_+$, where $\gamma_-$ is normally attracting, $\gamma_+$ is normally repelling and $p_0$ is a contact point (see Fig. \ref{fig-ergodic-model}).\\

\begin{figure}[htb]
	\begin{center}
		\includegraphics[width=4.4cm,height=3.5cm]{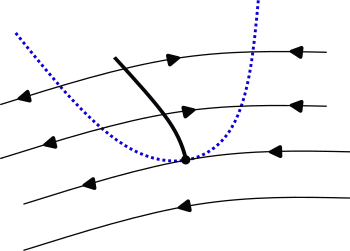}
  {\footnotesize
\put(-47,85){$\gamma_-$}
\put(-110,76){$\gamma_+$}
\put(-67,28){$p_0$}
\put(-73,67){$\sigma$}
  }
		         \end{center} 
	\caption{Dynamics of $X_{\lambda_0,0}$, with contact point $p_0$ separating normally attracting branch $\gamma_-$ and normally repelling branch $\gamma_+$.}
	\label{fig-ergodic-model}
\end{figure}
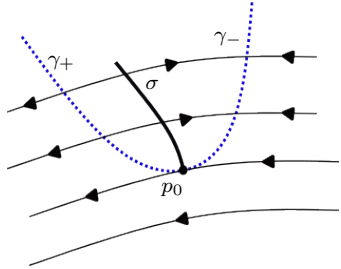

In Example \ref{example-1} (Section \ref{preliminaries-slow-fast}) Assumption 1 is satisfied if, for instance,
\begin{equation}\label{eq-conditions-up}
f_{\lambda_0}(0)=f_{\lambda_0}'(0)=0, \ f_{\lambda_0}'(x)>0 \text{ for } x>0, \ f_{\lambda_0}'(x)<0 \text{ for } x<0. 
\end{equation}
The contact point $p_0$ is given by $(x,y)=(0,0)$, $\gamma_-=\{\cS_{\lambda_0}\,|\,x>0\}$ and $\gamma_+=\{\cS_{\lambda_0}\,|\,x<0\}$.

\begin{remark}\label{remark-even}
    From Assumption 1 it follows that the contact order $\mathbbmss n$ (Section \ref{preliminaries-slow-fast}) of $p_0$ has to be even (when $\mathbbmss n$ is finite). Indeed, since $p_0\in \cS_{\lambda_0}$ is a nilpotent contact point for $\lambda=\lambda_0$ (see Assumption 1), there exist smooth local coordinates $( x, y)$ such that $p_0=(0,0)$ in which, up to multiplication by a strictly positive function, the slow-fast system $X_{\lambda,\epsilon}$ in \eqref{model} with $(\epsilon,\lambda)\sim(0,\lambda_0)$ can be written as

    \rev{
\begin{equation}\label{normal-form}
    \begin{cases}
        \dot{x}=y-f_\lambda(x)\\
        \dot{y} = \epsilon\left(g(x,\lambda,\epsilon)+ ( y-f_\lambda( x)) h( x, y,\lambda,\epsilon)\right),
    \end{cases}
\end{equation}
} where $f_\lambda$ and $g$ are given in \eqref{eq:sf2}, $h$ is a smooth function and $f_{\lambda_0}(0)=f_{\lambda_0}'(0)=0$ (see \cite[Proposition 2.1]{de2021canard}). Thus, \eqref{normal-form} is a normal form for smooth equivalence. Following \cite[Section 2.2]{de2021canard}, we can read the contact order of $p_0$ and the singularity order of $p_0$ from the normal form \eqref{normal-form}: $\mathbbmss n\ge 2$ is the order of the function $f_{\lambda_0}(x)$ at $x = 0$ and $\mathbbmss m\ge 0 $ is the order of $g(x,\lambda_0,0)$ at $x = 0$ (this is independent of the choice
of coordinates for the normal form \eqref{normal-form}). Now, since $p_0$ separates the attracting portion $\gamma_-\subset \cS_{\lambda_0}$ and the repelling portion $\gamma_+\subset\cS_{\lambda_0}$ (Assumption 1), it is clear that $f_{\lambda_0}'(x)\ne 0$ for $x\ne 0$ and $f_{\lambda_0}'(x)$ changes sign as one varies $x$ through $0$. Thus, $\mathbbmss n$ is even or $\mathbbmss n=\infty$, and $\cS_{\lambda_0}$ is a ``parabola-like" curve of singularities (see Fig. \ref{fig-ergodic-model}).
\end{remark}

In order to avoid any confusion we shall distinguish two cases when we use the slow-fast Li\'{e}nard equation $Y_{\lambda,\epsilon}$ in \eqref{eq:sf2}: the local case where $Y_{\lambda,\epsilon}$ appears in the normal form \eqref{normal-form} ($Y_{\lambda,\epsilon}$ is defined in a small neighborhood of the contact point $p_0=(0,0)$) and the global case where $Y_{\lambda,\epsilon}$ is defined on open set $M\subset \mathbb R^2$, often $M= \mathbb R^2$ (see Section \ref{subsection-limit cycles}, Section \ref{subsection-densities} and Section \ref{section-numerics}). In the global case we always assume that the contact point $p_0$ is located at the origin in the $(x,y)$-space and that \eqref{eq-conditions-up} holds.

Using Assumption 1 it is also clear that the slow vector field $\hat Q_{\lambda_0}(p)$ is well-defined for all $p\in \gamma_-\cup\gamma_+$ (see Section \ref{preliminaries-slow-fast}).
\smallskip

The next assumption deals with the singularity order of $p_0$.\\
\\
\textbf{Assumption 2} We suppose that the singularity order $\mathbbmss m$ of the contact point $p_0$ is finite and odd.

\begin{remark}\label{remark-direction}
Assumption 2 and Remark \ref{remark-even} imply that the slow vector field $\hat Q_{\lambda_0}$ points from $\gamma_-$ to $\gamma_+$ or from $\gamma_+$ to $\gamma_-$, near the contact point $p_0$ (hence, it is not directed towards $p_0$ or away from $p_0$ on both sides of $p_0$). To see this, it suffices to use the normal form \eqref{normal-form} near $p_0$. It can be easily seen that the slow vector field associated to \eqref{normal-form} is given by \eqref{eq-SVF-example} with $\lambda=\lambda_0$, $x\sim 0$ and $x\ne 0$. defined near $p_0$. Let us focus on the $x$-component of \eqref{eq-SVF-example}: 
\begin{equation}\label{SVF-x}\frac{\dd x}{\dd\tau}=\frac{g(x,\lambda_0,0)}{f_{\lambda_0}'(x)},
\end{equation}
with $x\ne 0$ and $x\sim 0$. Since the order of the function $g(x,\lambda_0,0)$ at $x = 0$ is finite and odd (Assumption 2 and Remark \ref{remark-even}), $g(x,\lambda_0,0)$ changes sign as $x$ goes through the origin. Recall that $f_{\lambda_0}'$ has the same property (Remark \ref{remark-even}). Thus, the right-hand side of \eqref{SVF-x} is either positive for all $x\ne 0$ and $x\sim 0$ or negative for all $x\ne 0$ and $x\sim 0$. 
\end{remark}
We further assume:\\
\\
\textbf{Assumption 3} $\gamma_-\cup\gamma_+$ does not contain singularities of the slow vector field $\hat Q_{\lambda_0}$, and $\hat Q_{\lambda_0}$ points from $\gamma_-$ to $\gamma_+$.  \\
\\
Assumption 3 is natural because in Section \ref{subsection-limit cycles} and Section \ref{subsection-densities} we study ergodic properties and entry-exit probability measures related to canard orbits of $X_{\lambda,\epsilon}$ with $\lambda\sim\lambda_0$ and $\epsilon$ small and positive. Such orbits follow a portion of the attracting curve $\gamma_-$, pass close to the contact point $p_0$ and then follow the repelling curve $\gamma_+$ for a significant amount of time. 

Since $\hat Q_{\lambda_0}$ is regular on $\gamma_-\cup\gamma_+$ (Assumption 3), the slow divergence integral associated to any segment contained in $\gamma_-\cup\gamma_+$ is well-defined (see Section \ref{preliminaries-slow-fast}). As mentioned before, it is important to work with the slow divergence integral associated to segments of $\cS_{\lambda_0}$ with the property that one of their endpoints is the contact point $p_0$. The following assumption enables us to extend the slow divergence integral to $p_0$, for $\lambda=\lambda_0$ (see Remark \ref{remark-extension}):\\
\\
\textbf{Assumption 4} We assume that $\mathbbmss m< 2(\mathbbmss n-1)$.

\begin{remark}
    \label{remark-extension} Suppose that $\lambda=\lambda_0$. Let $\gamma=[q,p_0]$ be a segment contained in $ \gamma_-\cup\{p_0\}$, with one of the endpoints equal to the contact point $p_0$. Then the slow divergence integral associated to $[q,p_0]$ is defined as 
    \begin{equation}\label{eq-SDI-attr} I_-([q,p_0]):=\lim_{p\to p_0,p\in\gamma_-}I([q,p],\lambda_0)<0,
    \end{equation}
    where $I$ is defined in \eqref{SDI-classicaldef} and associated to the normally attracting segment $[q,p]\in\gamma_-$. Using Assumption 4 and \cite[Definition 5.2]{de2021canard}, $I_-([q,p_0])$ is finite. This can be easily seen if we use the normal form \eqref{normal-form} near $p_0$ (recall that \eqref{SDI-classicaldef} is invariant under smooth equivalences). We may assume that the curve of singularities $y=f_{\lambda_0}(x)$ of \eqref{normal-form} satisfies \eqref{eq-conditions-up} near $x=0$ (if not, we can apply $(x,y)\to(-x,-y)$ to \eqref{normal-form}). Let $0<x_1<x_2$ with $x_2$ small. The slow divergence integral of \eqref{normal-form} associated to the attracting segment parameterized by $x\in [x_1,x_2]$ reads as
\begin{equation}\label{SDI-intr2}I(x_1,x_2)=-\int_{x_2}^{x_1}\frac{\left(f_{\lambda_0}'(x)\right)^2}{g(x,\lambda_0,0)}\dd x<0.\nonumber 
\end{equation}
For more details we refer to \cite[Section 5.5]{de2021canard} (see also \eqref{SDI-added}). Now, from Assumption 4 it follows that the following limit is finite:
\begin{equation}\label{eq:slow_I-att}
    I_-(x_2)=\lim_{x_1\to 0^+}I(x_1,x_2)=-\int_{x_2}^{0}\frac{\left(f_{\lambda_0}'(x)\right)^2}{g(x,\lambda_0,0)}\dd x<0. 
\end{equation} The integral $I_-(x_2)$ in \eqref{eq:slow_I-att} represents the slow divergence integral associated to the segment $[0,x_2]$ contained in $y=f_{\lambda_0}(x)$ where the endpoint $x=0$ corresponds to the contact point $(x,y)=(0,0)$. Thus, $I_-([q,p_0])$ in \eqref{eq-SDI-attr} is well-defined (i.e. finite).

Similarly, if $[q,p_0]$ is a segment contained in $ \gamma_+\cup\{p_0\}$, then we define 
    \begin{equation}\label{eq-SDI-rep} I_+([q,p_0]):=\lim_{p\to p_0,p\in\gamma_+}I([q,p],\lambda_0)>0.
    \end{equation} 
    In the normal form coordinates we have 
    \begin{equation}\label{eq:slow_I-rep}
    I_+(x_1)=\lim_{x_2\to 0^-}I(x_1,x_2)=-\int_{0}^{x_1}\frac{\left(f_{\lambda_0}'(x)\right)^2}{g(x,\lambda_0,0)}\dd x>0,
\end{equation}
where $x_1<x_2<0$.
\end{remark}
We finally define the notion of slow relation function of \eqref{model} for $\lambda=\lambda_0$. 
Let $\sigma\subset M$ be a smooth closed section transverse to the fast foliation $\mathcal F_{\lambda_0}$, having the contact point $p_0$ as its endpoint (Fig. \ref{fig-ergodic-model}). We let $\sigma$ be parameterized by a regular parameter $s\in [0,s_0]$, with $s_0>0$, where $s=0$ corresponds to $p_0$, and we suppose that $\sigma\setminus\{p_0\}$ lies  
in the basin of attraction of $\gamma_-$ and, in backward time,
in the basin of attraction of $\gamma_+$. We write 
\begin{equation}
    \label{SDI-papameter-s}
    \tilde I_-(s)=I_-([\omega(s),p_0]), \  \tilde I_+(s)=I_+([\alpha(s),p_0]), \ s\in ]0,s_0],
\end{equation}
where $I_\pm([q,p_0])$ are defined in \eqref{eq-SDI-attr} and \eqref{eq-SDI-rep} and $\omega (s)\in\gamma_-$ (resp. $\alpha (s)\in \gamma_+$) is the $\omega$-limit point (resp. $\alpha$-limit point) of the orbit of $X_{\lambda_0,0}$ through $s\in \sigma $. It is clear that $\tilde I_\pm(s)\to 0$ as $s$ tends to zero, $\tilde I_-$ is strictly decreasing and smooth on $]0,s_0]$ ($\tilde I_-'(s)<0$ for $s\in ]0,s_0]$) and $\tilde I_+$ is strictly increasing and smooth on $]0,s_0]$ ($\tilde I_+'(s)>0$ for $s\in ]0,s_0]$). If we take $\tilde I_\pm(0)=0$, then the functions $\tilde I_\pm$ are continuous on the segment $[0,s_0]$.
\begin{definition}[\textbf{Slow-relation function}]\label{def-1}
Consider $X_{\lambda,\epsilon}$ defined in \eqref{model} and suppose that Assumptions 1 through 4
are satisfied.  If $$-\tilde I_-(s_0)\le \tilde I_+(s_0) \text{ (resp. } -\tilde I_-(s_0)> \tilde I_+(s_0)),$$ then $S:[0,s_0]\to [0,s_0]$, $S(0)=0$, given by
\begin{equation}
    \label{slow-relation-function}
    \tilde I_-(s)+ \tilde I_+(S(s))=0 \text{ (resp. } \tilde I_-(S(s))+ \tilde I_+(s)=0), \ s\in ]0,s_0],
\end{equation}
is well-defined and we call it \emph{the slow relation function}. 
\end{definition}
\begin{remark}\label{rem-existence}
    Let us explain why the function $S$ in Definition \ref{def-1} is well-defined (i.e. $S$ exists). Suppose that $-\tilde I_-(s_0)\le \tilde I_+(s_0)$ and take any $s\in ]0,s_0]$. Since $-\tilde I_-$ is strictly increasing and $-\tilde I_-(0)=0$, we have $0<-\tilde I_-(s)\le -\tilde I_-(s_0)$. Thus, $0=\tilde I_+(0)<-\tilde I_-(s)\le \tilde I_+(s_0)$. The function $\tilde I_+$ is continuous on the segment $[0,s_0]$, and the Intermediate-Value Theorem implies the existence of a unique number $S(s)$ in $]0,s_0]$ such that $\tilde I_+(S(s))=-\tilde I_-(s)$ (the uniqueness and $S(s)>0$ follow from the fact that $\tilde I_+$ is strictly increasing). The case where $-\tilde I_-(s_0)> \tilde I_+(s_0)$ can be treated in similar fashion as above.  
\end{remark}

Since $S(0)=0$ and $S(s)\to 0$ as $s$ tends to zero, it is clear that the slow relation function $S$ is continuous on $[0,s_0]$. Moreover, the Implicit Function Theorem, the smoothness of $\tilde I_\pm$ on the interval $]0,s_0]$ and \eqref{slow-relation-function} imply the smoothness of $S$ on the interval $]0,s_0]$. Moreover, $S'>0$.

\begin{remark}\label{remark-borel-sigma}
    In Section \ref{subsection-limit cycles} we assume that $[0,s_0]$ is a segment on $\mathbb R$ with the standard Borel $\sigma$-algebra and work with $S$-invariant Borel probability measures on $[0,s_0]$. The main results in Section \ref{subsection-limit cycles} (Theorem \ref{theorem-1}--Theorem \ref{theorem-3})
     are independent of the choice of section $\sigma$ and a regular parameter $s$ on $\sigma$. 

     In Section \ref{subsection-densities} we study connection between entry and exit probability measures and it is natural to deal with a more general definition of $S$. Instead of one section $\sigma$ we have two sections $\sigma_-$ (entry) and $\sigma_+$(exit). We refer to Fig \ref{fig-entry-exit}. 
\end{remark}

 We say that the multiplicity of a fixed point $s_1\in ]0,s_0]$ of $S$ is equal to $l$ if $s_1$ is a zero of $\tilde S(s):=s-S(s)$ of multiplicity $l$ (that is $\tilde S(s_1)=\cdots=\tilde S^{(l-1)}(s_1)=0$ and $\tilde S^{(l)}(s_1)\ne 0$). If $\tilde S^{(n)}(s_1)=0$ for each $n=0,1,\dots$, then the multiplicity of $s_1$ of $S$ is $\infty$.

 \begin{remark}
    We will often work with slow relation functions associated to slow-fast Li\'{e}nard systems \eqref{eq:sf2} satisfying \eqref{eq-conditions-up} (see Section \ref{subsection-limit cycles}, Section \ref{subsection-densities} and Section \ref{section-numerics}). In this case we can take  $\sigma\subset \{x=0\}$, parameterized by the coordinate $y\in [0,s_0]$. We denote $y$ by $s$. Then the integrals $\tilde I_\pm(s)$ in \eqref{SDI-papameter-s} become
    \begin{equation}
        \label{integrals-Lienard}
         \tilde I_-(s)=-\int_{\omega_1(s)}^{0}\frac{\left(f_{\lambda_0}'(x)\right)^2}{g(x,\lambda_0,0)}\dd x, \ \tilde I_+(s) =-\int_{0}^{\alpha_1(s)}\frac{\left(f_{\lambda_0}'(x)\right)^2}{g(x,\lambda_0,0)}\dd x,
    \end{equation}
 where $\alpha_1(s)<0$ and $\omega_1 (s)>0$ are the $x$-coordinates of the $\alpha$ and $\omega$ limits of the fast orbit through $s$ (see \eqref{eq:slow_I-att} and \eqref{eq:slow_I-rep}). We have $s=f_{\lambda_0}(\alpha_1(s))$ and $s=f_{\lambda_0}(\omega_1(s))$, and by differentiating it follows that
 $$1=f_{\lambda_0}'(\alpha_1(s))\alpha_1'(s), \qquad 1=f_{\lambda_0}'(\omega_1(s))\omega_1'(s).$$
This previous equation, together with \eqref{integrals-Lienard}, imply that
\begin{equation}
    \label{derivative-Lienard}
 \tilde I_-'(s)=\frac{f_{\lambda_0}'(\omega_1(s))}{g(\omega_1(s),\lambda_0,0)}, \qquad \tilde I_+'(s) =-\frac{f_{\lambda_0}'(\alpha_1(s))}{g(\alpha_1(s),\lambda_0,0)}.    
\end{equation}
 \end{remark}

\subsection{Invariant measures and limit cycles}\label{subsection-limit cycles}
In this section, we suppose that $X_{\lambda,\epsilon}$ in \eqref{model} satisfies Assumption 1--Assumption 4.
For each $s\in ]0,s_0]$ we define a closed curve $\Gamma_s$ at level $(\lambda,\epsilon)=(\lambda_0,0)$ consisting of the fast orbit of $X_{\lambda_0,0}$ passing through $s\in\sigma$ and the portion of the curve of singularities $\cS_{\lambda_0}$ between the $\omega$-limit point $\omega(s)\in\gamma_-$ and the $\alpha$-limit point $\alpha(s)\in\gamma_+$ of that fast orbit (see Fig. \ref{fig-motivation-introduction}(b)). We associate the following slow divergence integral to $\Gamma_s$:
\begin{equation}
    \label{SDI-final-version}
   \tilde I(s):= \tilde I_-(s)+ \tilde I_+(s),
\end{equation}
with $s\in ]0,s_0]$.
\begin{theorem}
    \label{theorem-1}
    Let $S:[0,s_0]\to [0,s_0]$ be the slow relation function defined in \eqref{slow-relation-function} and let $\tilde I$ be the slow divergence integral associated to $\Gamma_s$, defined in \eqref{SDI-final-version}. Then the following statements hold:
    \begin{enumerate}
        \item The function $\tilde I$ has no zeros in $]0,s_0]$ if and only if the slow relation function $S$ is uniquely ergodic (i.e. $S$ admits
precisely one invariant probability measure: the Dirac delta
measure $\delta_0$ at $0$). 
        \item The function $\tilde I$ has a zero at $s=s_1\in ]0,s_0]$ if and only if the Dirac delta measure $\delta_{s_1}$ at the point $s_1$ is $S$-invariant.
        \item The function $\tilde I$ has  exactly $k$ zeros $s_1<\dots <s_k$ in $]0,s_0]$ if and only if the set of all $S$-invariant probability measures on $[0,s_0]$, denoted by $\mathcal P_S$, is the convex hull of Dirac delta measures $\delta_0,\delta_{s_1},\dots,\delta_{s_k}$:
        \begin{equation}\label{convex-set}
            \mathcal P_S=\left\{\eta_0\delta_0+\sum_{i=1}^k\eta_i\delta_{s_i}: \eta_0,\,\eta_i\ge 0, \,\sum_{i=0}^k\eta_i=1 \right\}.
        \end{equation}
    \end{enumerate}
\end{theorem}
We prove Theorem \ref{theorem-1} in Section \ref{section-proof-1}. We point out that $k$ in Theorem \ref{theorem-1}.3 is the arithmetic number of zeros of $\tilde I$, i.e. the zeros of $\tilde I$ counted without their multiplicity. Notice that $\delta_0,\delta_{s_1},\dots,\delta_{s_k}$ from Theorem \ref{theorem-1}.3 are  ergodic probability measures (they are the extremal
points of the convex set $\mathcal P_S$ in \eqref{convex-set}). See also \cite[Proposition 4.3.2]{viana2016foundations}. In the proof of Theorem \ref{theorem-1}.1 and Theorem \ref{theorem-1}.3 we use an important result in ergodic theory, the Poincar\'{e} recurrence theorem \cite{katok,viana2016foundations}.

\rev{ We point out that the study of zeros of $\tilde I$ is relevant since the zeros provide candidates for limit cycles (for more details see Theorem \ref{theorem-2} and Theorem \ref{theorem-3} and their proof).}

\begin{example}\label{example-vanderpol} Consider the slow-fast Van der Pol system 
\rev{
\begin{equation}\label{example-classical-Lienard}
    X_{\lambda,\epsilon}:\begin{cases}
        \dot{x}=y-\frac{1}{2}x^2-\frac{1}{3}x^3\\
        \dot{y} = \epsilon \left (\lambda-x\right),
    \end{cases}
\end{equation}
}
where $\lambda\sim 0$ ($\lambda_0=0$). The slow relation function associated with the slow-fast system \eqref{example-classical-Lienard} is uniquely ergodic. Indeed, for $\epsilon=\lambda=0$, we consider the normally attracting branch $\gamma_-=\{y=\frac{1}{2}x^2+\frac{1}{3}x^3\}\cap \{x>0\}$, the normally repelling branch $\gamma_+=\{y=\frac{1}{2}x^2+\frac{1}{3}x^3\}\cap \{-1<x<0\}$ and the contact point $p_0$ at $(x,y)=(0,0)$. Note that \eqref{example-classical-Lienard} is a special case of \eqref{eq:sf2}. We take $s=y\in [0,s_0]$, where $s_0\in ]0,\frac{1}{6}[$ is arbitrary and fixed. Using \eqref{integrals-Lienard}, the slow divergence integral in \eqref{SDI-final-version} can be written as 
$$\tilde I(s)=-\int_{\alpha_1(s)}^{\omega_1 (s)}x(1+x)^2\dd x, \qquad s\in ]0,s_0]. $$
 Since $\tilde I(s)<0$ for all $s\in ]0,s_0]$ (see \cite{1996} or \cite[Section 5.7]{de2021canard}), Theorem \ref{theorem-1}.1 implies that the slow relation function $S:[0,s_0]\to [0,s_0]$ is uniquely ergodic.

We call the contact point $p_0$ in \eqref{example-classical-Lienard} a slow-fast Hopf point (see below).\end{example}

\begin{example}
 Consider \eqref{eq:sf2} with $f_\lambda(x)=x^\mathbbmss n$ and $g(x,\lambda,\epsilon)=-x^\mathbbmss m$, where $\mathbbmss n\ge 2$ is even, $\mathbbmss m\ge 1$ is odd and $\mathbbmss m< 2(\mathbbmss n-1)$. Since the function $f_\lambda$ is even (i.e. the curve of singularities is symmetric w.r.t. the $y$-axis) and the function $x\mapsto \frac{\left(f_{\lambda}'(x)\right)^2}{g(x,\lambda,0)}$ is odd, the slow relation function $S$ is the identity map and the slow divergence integral $\tilde I$ is identically zero. In this case, each probability measure is $S$-invariant, and ergodic probability measures are given by Dirac delta measures. 
\end{example}
For slow-fast Li\'{e}nard equations with arbitrary number of zeros of the associated slow divergence integral we refer to e.g. \cite{SDICLE1}.

\smallskip

Assume that the contact point $p_0$ in Assumption 1 is of Morse type (this means that the contact order of $p_0$ is $2$) and that the singularity order of $p_0$ is $1$. If the slow vector field $\hat Q_{\lambda_0}$, defined in Section \ref{subsection-model}, points from the attracting  branch $\gamma_-$ to the repelling branch $\gamma_+$, then we say that $X_{\lambda,\epsilon}$ has a slow-fast Hopf point at $p_0$ for $\lambda=\lambda_0$ (sometimes called generic turning point). See e.g. \cite{de2021canard,kuehn2015multiple}. When $p_0$ is a slow-fast Hopf point, then $\Gamma_s$ (often called a canard cycle) can produce limit cycles after perturbation. More precisely, we say that the cyclicity of the canard cycle $\Gamma_s$ is bounded by $N\in\mathbb{N}_0$ if there exist $\epsilon_0>0$, $\delta_0>0$ and a neighborhood $\mathcal V$ of $\lambda_0$ in the $\lambda$-space such that $X_{\lambda,\epsilon}$ has at most $N$ limit cycles lying within Hausdorff
distance $\delta_0$ of $\Gamma_s$ for each $(\lambda,\epsilon)\in \mathcal V\times [0,\epsilon_0]$. The smallest $N$ with this property is called the cyclicity of $\Gamma_s$. We denote by $\cycl(X_{\lambda,\epsilon},\Gamma_s)$ the cyclicity of $\Gamma_s$. We have

\begin{theorem}
    \label{theorem-2}
    Suppose that $X_{\lambda,\epsilon}$ has a slow-fast Hopf point at $p_0$ for $\lambda=\lambda_0$. Let $S:[0,s_0]\to [0,s_0]$ be the slow relation function from Definition \ref{def-1}, associated to $X_{\lambda,\epsilon}$. The following statements are true.
    \begin{enumerate}
        \item If $S$ is uniquely ergodic, then $\cycl(X_{\lambda,\epsilon},\Gamma_s)\le 1$ for each fixed $s\in]0,s_0]$. The limit cycle, if it exists Hausdorff close to $\Gamma_s$, is hyperbolic and attracting (resp. repelling) if $\tilde I(s)<0$ (resp. $>0$).
        \item If the Dirac delta measure $\delta_{s_1}$ is $S$-invariant for some $s_1\in ]0,s_0]$, then $s=s_1$ is a fixed point of $S$ of multiplicity $1\le l\le \infty$, and $\cycl(X_{\lambda,\epsilon},\Gamma_{s_1})\le l+1$ if $l<\infty$.
    \end{enumerate}

\end{theorem}
Theorem \ref{theorem-2} will be proved in Section \ref{section-proof-2}.

\smallskip

Theorem \ref{theorem-3} below deals with the following generalization of slow-fast Hopf points:
\rev{
\begin{equation}\label{system-non-generic}
    \begin{cases}
        \dot{x}=y-x^{2\mathbbmss n_1}\tilde f(x)\\
        \dot{y} = \epsilon \left (\lambda-x^{2\mathbbmss n_1-1}\right),
    \end{cases}
\end{equation}
}
where $\tilde f$ is smooth, $\tilde f(0)>0$, $\mathbbmss n_1\ge 1$ and $\lambda\sim 0\in\mathbb R$ ($\lambda_0=0$). The Li\'{e}nard equation \eqref{system-non-generic} is of type \eqref{eq:sf2} with $f_\lambda (x)=x^{2\mathbbmss n_1}\tilde f(x)$ and $g(x,\lambda,\epsilon)=\lambda-x^{2\mathbbmss n_1-1}$. For $\lambda=0$, the origin $(x,y)=(0,0)$ is a contact point with even contact order $2\mathbbmss n_1$ and odd singularity order $2\mathbbmss n_1-1$. It is clear that Assumption 2 and Assumption 4 are satisfied. When $\mathbbmss n_1=1$ (resp. $\mathbbmss n_1>1$), $(x,y)=(0,0)$ is a slow-fast Hopf point or generic turning point (resp. a non-generic turning point). We suppose that 
\begin{equation}\label{assump-non-generic}
f_{\lambda_0}'(x)>0 \text{ for all } x>0, \ f_{\lambda_0}'(x)<0 \text{ for all } x<0.
\end{equation}
From \eqref{assump-non-generic} and $g(x,0,0)=-x^{2\mathbbmss n_1-1}$ it follows that Assumption 1, with $\gamma_-=\{y=x^{2\mathbbmss n_1}\}\cap \{x>0\}$ and $\gamma_+=\{y=x^{2\mathbbmss n_1}\}\cap \{x<0\}$, and Assumption 3 are satisfied. We define the slow relation function $S:[0,s_0]\to [0,s_0]$ of \eqref{system-non-generic} using \eqref{slow-relation-function}.

\begin{theorem}
    \label{theorem-3} Consider \eqref{system-non-generic} with a fixed $\mathbbmss n_1\ge 1$. If the set of all $S$-invariant probability measures is given by \eqref{convex-set} for some $0<s_1<\dots <s_k<s_0$, then $s_1,\dots,s_k$ are fixed points of $S$. If they all have multiplicity $1$ (i.e. they are hyperbolic) and if we take $s_k<s_{k+1}\le s_0$, then there exists a continuous function $\lambda_*(\epsilon)$, with ${\lambda}_*(0)=0$, such that the Li\'{e}nard family \eqref{system-non-generic} with $\lambda=\lambda_*(\epsilon)$ has $k+1$ periodic orbits ${O}_1^\epsilon,\dots,{O}_{k+1}^\epsilon$, for each $\epsilon\sim 0$ and $\epsilon>0$. The periodic orbit ${O}_i^\epsilon$ is isolated, hyperbolic and close to $\Gamma_{s_i}$ in Hausdorff sense, for each $i=1,\dots, k+1$.
\end{theorem}
Theorem \ref{theorem-3} will be proved in Section \ref{section-proof-3}.

\subsection{Entry and exit measures}\label{subsection-densities}
In this section we deal with Borel probability measures on $\mathbb R$ with the usual Borel $\sigma$-algebra $\mathcal B$. The measures will be supported on bounded Borel sets $L, T, \dots$ (see below). Roughly speaking, the main result of this section, Theorem \ref{mainthm-weak}, gives an answer to following natural questions: if an $\epsilon$-family of entry measures ($\epsilon$ is the singular perturbation parameter) is convergent when $\epsilon\to 0$, is the $\epsilon$-family of exit measures convergent when $\epsilon\to 0$, and, if the exit limit exists, how do we read the exit limit from the entry limit and slow relation function? 

\begin{figure}[htb]
	\begin{center}
		\includegraphics[width=13.8cm,height=5.3cm]{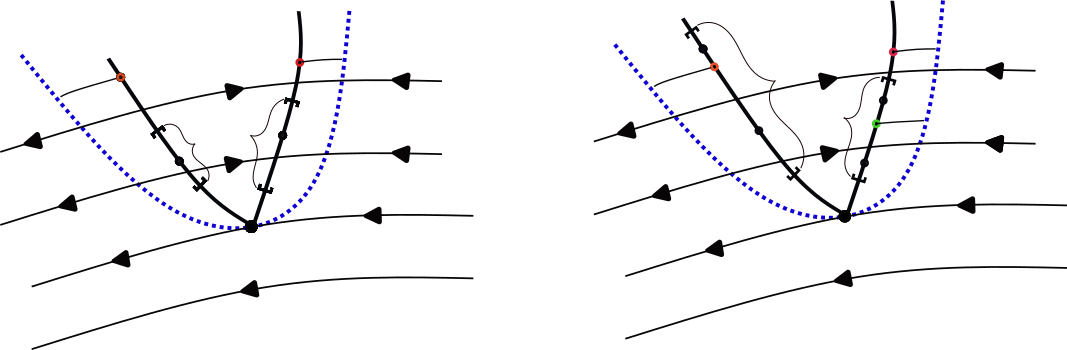}
  {\footnotesize
\put(-44,135){\large $\gamma_-$}
\put(-66,152){\large $\sigma_-$}
\put(-145,146){\large $\sigma_+$}
\put(-172,118){\large $\gamma_+$}
\put(-83,47){\large $p_0$}
\put(-91,97){\large $L$}
\put(-71,89){\large {$s_b^-$}}
\put(-73,76){\large{$s_1^-$}}
\put(-127,90){\large{$s_1^+$}}
\put(-66,105){\large{$s_2^-$}}
\put(-148,124){\large{$s_2^+$}}
\put(-128,120){\large{$s_c^+$}}
\put(-77,125){\large{$s_c^-$}}
\put(-107,114){\large $T$}
\put(-261,132){\large $\gamma_-$}
\put(-285,148){\large $\sigma_-$}
\put(-356,129){\large $\sigma_+$}
\put(-391,115){\large $\gamma_+$}
\put(-299,44){\large $p_0$}
\put(-309,88){\large $L$}
\put(-320,87){\large $T$}
\put(-346,116){\large{$s_c^+$}}
\put(-296,120){\large{$s_c^-$}}
\put(-286,90){\large{$s_1^-$}}
\put(-341,76){\large{$s_1^+$}}
\put(-332,-5){\large (a) $-\tilde I_-(s_c^-)\le \tilde I_+(s_c^+)$}
\put(-128,-5){\large (b) $-\tilde I_-(s_c^-)>\tilde I_+(s_c^+)$}
  }
		         \end{center} 
	\caption{Entry section  $\sigma_-$, exit section $\sigma_+$ and definition of slow relation function $S_0:L\to T$.}
	\label{fig-entry-exit}
\end{figure}
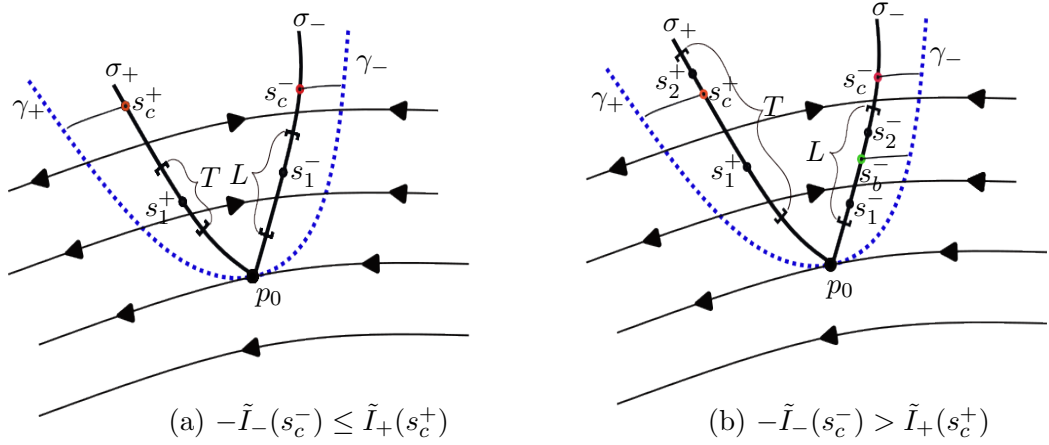

We consider $X_{\lambda,\epsilon}$ defined in \eqref{model} and suppose that it satisfies Assumption 1--Assumption 4. Instead of one section $\sigma$ (Section \ref{subsection-limit cycles}) we now define two sections $\sigma_-$ and $\sigma_+$, transverse to the fast foliation $\mathcal F_{\lambda_0}$. We refer to Fig. \ref{fig-entry-exit}. We parameterize $\sigma_\pm$ by a regular parameter $s^\pm\in [0,s_0^\pm]$, where $s^\pm=0$ corresponds to the contact point $p_0$. The section $\sigma_-\setminus\{p_0\}$ lies  
in the basin of attraction of $\gamma_-$ and the section $\sigma_+\setminus\{p_0\}$, in backward time,
in the basin of attraction of $\gamma_+$.
\smallskip

Again we can define slow divergence integrals $\tilde I_-(s^-)<0$ and $\tilde I_+(s^+)>0$ (see \eqref{SDI-papameter-s}). We take a point on $\sigma_-$ given by $s_c^-\in ]0,s_0^-[$ and a point on $\sigma_+$ given by $s_c^+\in ]0,s_0^+[$. We distinguish between two cases:
\begin{itemize}
    \item[(a)]{$-\tilde I_-(s_c^-)\le \tilde I_+(s_c^+)$.} For any segment $L$ contained in $]0,s_c^-[$, we can define a smooth and increasing slow relation function $S_0:L\to T=S_0(L)\subset ]0,s_c^+[$ by $\tilde I_-(s^-)+ \tilde I_+(S_0(s^-))=0$, $s^-\in L$. The proof that $S_0$ is well-defined is similar to the proof that $S$ is well-defined, given in Remark \ref{rem-existence}. See Fig. \ref{fig-entry-exit}(a). If $\sigma_-=\sigma_+$, $s^-=s^+$ and $s_c^-=s_c^+$, then $S_0$ is the slow relation function $S$ from Definition \ref{def-1}, defined on $[0,s_c^-]$.
    \item[(b)]{$-\tilde I_-(s_c^-)>\tilde I_+(s_c^+)$.} In this case there exists a unique $s_b^-\in ]0,s_c^-[$ such that $\tilde I_-(s_b^-)+ \tilde I_+(s_c^+)=0$ ($s_b^-$ is called a buffer point \cite[Section 7.4]{de2021canard}). For $s_1^-\in ]0,s_b^-[$, there is a unique $s_1^+\in ]0,s_c^+[$ such that $\tilde I_-(s_1^-)+ \tilde I_+(s_1^+)=0$. For $s_2^-\in ]s_b^-,s_c^-[$, there is a unique $s_2^+\in ]s_c^+,s_0^+]$ such that $\tilde I_-(s_2^-)+ \tilde I_+(s_2^+)=0$ (at least for $s_2^-$ close to $s_b^-$). We use a similar argument as in Remark \ref{rem-existence}. For a segment $L$ contained in $]0,s_c^-[$ and with $s_b^-$ in its interior, we consider (smooth and increasing) slow relation function $S_0:L\to T=S_0(L)\subset ]0,s_0^+]$ again defined by $\tilde I_-(s^-)+ \tilde I_+(S_0(s^-))=0$, $s^-\in L$. Clearly, $S_0(s_b^-)=s_c^+$. See Fig. \ref{fig-entry-exit}(b).   
\end{itemize}
\smallskip

In the case when a Borel probability measure $\mu_0$ (supported on $L$) has a density, then it is often important (see Section \ref{section-numerics}) to compute a density of the push-forward probability measure $\mu_0S_0^{-1}$.  It is well-known (see e.g. \cite[Section 3.2]{Lasota} or \cite[Theorem 11.8]{Sarapa}) that, if $D_{en}:\mathbb R\to \mathbb R$ is a density, supported on an interval $\overline L$, and $G:\mathbb R\to \mathbb R$ a Borel function such that $G:\overline L\to \overline T=G(\overline L)$ is bijective, $G^{-1}$ has a continuous derivative on $\overline T$ and $\frac{d}{ds}G^{-1}(s)\ne 0$ for all $s\in \overline T$, then $D_{en}$ is transformed by $G$ into a new density 
\begin{equation}\label{formula-density-trans}
   D_{ex}(s)= D_{en}(G^{-1}(s))\left | \frac{d}{ds}G^{-1}(s) \right |1_{\overline T}(s),
\end{equation} 
\rev{where $1_{\overline T}$ is the characteristic function of the set $\overline T$.}


\begin{proposition}\label{prop-Frob}
  Let $S_0:L\to T$ be a slow relation function defined in (a) or (b) and let $D_{en}:\mathbb R\to \mathbb R$ be an entry density supported on $L$. Then $D_{en}$ is transformed by $S_0$ into the following exit density supported on $T$:
  \begin{equation}\label{exit-formula-general}
      D_{ex}(s^+)=
                      -D_{en}(S_0^{-1}(s^+))\frac{\tilde I_+'(s^+)}{\tilde I_-'(S_0^{-1}(s^+))}1_T(s^+).
\end{equation}
  \end{proposition}
   \begin{proof}
Proposition \ref{prop-Frob} follows from \eqref{formula-density-trans} by taking $G=S_0$. Note that $S_0$ and $S_0^{-1}$ are smooth and increasing and that we can compute $\frac{d}{ds^+}S_0^{-1}(s^+)$ by using $\tilde I_-(s^-)+ \tilde I_+(S_0(s^-))=0$. \end{proof} 
\rev{
\begin{remark}
For a uniform entry density, the first factor in \eqref{exit-formula-general} is a constant. More precisely, if $D_{en}(s^-)=\frac{1}{|L|} 1_L(s^-)$, where $|L|$ denotes the length of the segment $L$, then \eqref{exit-formula-general} becomes 
$$D_{ex}(s^+)=
                      -\frac{1}{|L|}\frac{\tilde I_+'(s^+)}{\tilde I_-'(S_0^{-1}(s^+))}1_T(s^+).$$
 We use this formula in Example \ref{example-num-Pol} in Section \ref{section-numerics} when we compute exit densities for the van der Pol equation.                   
\end{remark}
}

    We can apply Proposition \ref{prop-Frob} to find exit densities in slow-fast Li\'{e}nard family \eqref{eq:sf2}. We can take $s^\pm$ to be the coordinate $y$.
\begin{corollary}
    \label{corollary-density}
    Suppose that \eqref{eq:sf2} satisfies Assumption 1--Assumption 4. Let $S_0:L\to T$ be a slow relation function associated to \eqref{eq:sf2} and let $D_{en}:\mathbb R\to \mathbb R$ be an entry density supported on $L$. Then $D_{en}$ is transformed by $S_0$ into 
    \begin{equation}\label{Frobenius-slow-fast-Lienard}
      D_{ex}(s^+)
      =                   D_{en}(S_0^{-1}(s^+))\frac{f_{\lambda_0}'(\alpha_1(s^+))g(\omega_1(S_0^{-1}
                      (s^+)),\lambda_0,0)}{f_{\lambda_0}'(\omega_1(S_0^{-1}
                     (s^+)))g(\alpha_1(s^+),\lambda_0,0)}1_T(s^+). \end{equation}
 
\end{corollary}
\begin{proof}
Expression \eqref{Frobenius-slow-fast-Lienard} follows directly from \eqref{derivative-Lienard} and \eqref{exit-formula-general}.\end{proof}

In the rest of this section we focus on 
\rev{
\begin{equation}\label{system-non-generic-measure}
    \begin{cases}
        \dot{x}=y-x^{2\mathbbmss n_1}\tilde f(x)\\
        \dot{y} = \tilde\epsilon^{2\mathbbmss n_1} \left (\tilde\epsilon^{2\mathbbmss n_1-1}\tilde \lambda-x^{2\mathbbmss n_1-1}\right),
    \end{cases}
\end{equation}
}
where $0\leq \tilde\epsilon\ll1$ is a new singular perturbation parameter, $\tilde\lambda\sim 0$ and $\tilde f$ is smooth with $\tilde f(0)>0$. Suppose that   Assumption 1--Assumption 4 are satisfied. Note that \eqref{system-non-generic-measure} is \eqref{system-non-generic} with $(\epsilon,\lambda)=(\tilde\epsilon^{2\mathbbmss n_1},\tilde\epsilon^{2\mathbbmss n_1-1}\tilde \lambda)$. For the sake of simplicity, we state Theorem \ref{mainthm-weak} for system \eqref{system-non-generic-measure} (the same result can be proved in a more general framework \cite{DM-entryexit}).

Following \cite[Theorem 7.7]{de2021canard} or \cite{DM-entryexit}, there exists a smooth curve $\tilde \lambda=\tilde\lambda_c(\tilde \epsilon)$
such that for every $\tilde\epsilon>0$ system \eqref{system-non-generic-measure}, with $\tilde \lambda=\tilde\lambda_c(\tilde \epsilon)$, has an orbit connecting $s_c^-\in \sigma_-$ with $s_c^+\in \sigma_+$. $\tilde \lambda=\tilde\lambda_c(\tilde \epsilon)$ is sometimes called a control curve. We denote by $S_{\tilde\epsilon}$, $\tilde\epsilon>0$, the transition map of \eqref{system-non-generic-measure}, with $\tilde \lambda=\tilde\lambda_c(\tilde \epsilon)$, from $L$ to $\sigma_+$. Clearly, $S_{\tilde\epsilon}:L\to S_{\tilde\epsilon}(L)$ is a smooth diffeomorphism and $S_{\tilde\epsilon}'>0$, due to the chosen parameterization of $\sigma_\pm$. The following result is a direct consequence of \cite[Proposition 7.1]{de2021canard} and \cite[Theorem 7]{DM-entryexit} (see also \cite[Section 3]{Dbalanced}).
\begin{proposition}\label{prop:slowrel}
Let $S_0:L\to T$ be a slow relation function associated to \eqref{system-non-generic-measure} and let $\tilde \lambda=\tilde\lambda_c(\tilde \epsilon)$ be a control curve as above. The following statements are true.
\begin{itemize}
    \item[(a)] If $-\tilde I_-(s_c^-)\le \tilde I_+(s_c^+)$, then for $\tilde\epsilon>0$ small enough the orbit through $s_1^-\in L$ (tunnel behavior) of system \eqref{system-non-generic-measure}, with $\tilde \lambda=\tilde\lambda_c(\tilde \epsilon)$, intersects $\sigma_+$ in positive time at $$s^+=S_{\tilde\epsilon}(s_1^-)=s^+_1+o(1), \ \tilde \epsilon\to 0,$$
    where $s_1^+=S_0(s_1^-)$ and $o(1)$ tends to $0$ as $\tilde \epsilon\to 0$, uniformly in $L$. 
    \item[(b)] If $-\tilde I_-(s_c^-)>\tilde I_+(s_c^+)$, then for $\tilde\epsilon>0$ small enough the orbit through $s_1^-\in L\cap ]-\infty,s_b^-[$ (tunnel behavior) (resp. $s_2^-\in L\cap ]s_b^-,+\infty[$ (funnel behavior)) of system \eqref{system-non-generic-measure}, with $\tilde \lambda=\tilde\lambda_c(\tilde \epsilon)$, intersects $\sigma_+$ in positive time at $$s^+=S_{\tilde\epsilon}(s_1^-)=s^+_1+o(1), \ \tilde \epsilon\to 0,$$
    where $s_1^+=S_0(s_1^-)$ and $o(1)$ tends to $0$ as $\tilde \epsilon\to 0$, uniformly in any compact subset of $L\cap ]-\infty,s_b^-[$  (resp. $s^+=S_{\tilde\epsilon}(s_2^-)=s^+_c+o(1)$, $\tilde \epsilon\to 0$).
\end{itemize}
\end{proposition}
Following Proposition \ref{prop:slowrel}, in the tunnel region the transition map $S_{\tilde\epsilon}$ is a small ${\tilde\epsilon}$-perturbation of the slow relation function $S_0$, while in the funnel region $S_{\tilde\epsilon}$ is close to the constant $s^+_c$. In Proposition \ref{prop:slowrel}(b) these two regions are separated by the buffer point $s_b^-$. 

If $\mu_{\tilde\epsilon},\mu_0$ are probability measures supported on $L$, then $\mu_{\tilde\epsilon}S_{\tilde\epsilon}^{-1},\mu_0S_0^{-1}$ denote push-forward probability measures of $\mu_{\tilde\epsilon},\mu_0$. Notice that $\mu_{\tilde\epsilon}S_{\tilde\epsilon}^{-1},\mu_0S_0^{-1}$ are supported on $S_{\tilde\epsilon}(L),S_0(L)=T$. Assume that $\mu_{\tilde\epsilon}$ converges weakly to $\mu_0$ as $\tilde \epsilon\to 0$. In the first case (Fig. \ref{fig-entry-exit}(a)), we show that $\mu_{\tilde\epsilon}S_{\tilde\epsilon}^{-1}$ converges weakly to $\mu_0S_0^{-1}$ as $\tilde \epsilon\to 0$ (see Theorem \ref{mainthm-weak}(a) below). In the second case (Fig. \ref{fig-entry-exit}(b)), $\mu_{\tilde\epsilon}S_{\tilde\epsilon}^{-1}$ converges weakly to $\mu_0\widetilde S_0^{-1}$ as $\tilde \epsilon\to 0$, where $\widetilde S_0:L\to \widetilde S_0(L)=T\cap ]-\infty,s_c^+]$ is a continuous function defined by 
    \begin{equation}\label{mixture-measure-push-forward}
    \widetilde S_0(s^-)=\begin{cases}
        S_0(s^-), \ \ s^-\in L\cap ]-\infty,s_b^-[, \\
        s_c^+, \qquad \  s^-\in L\cap [s_b^-,+\infty[.
    \end{cases}
    \end{equation}
    We refer to Theorem \ref{mainthm-weak}(b). Notice that the function $\widetilde S_0$ is equal to the slow relation function $S_0$ below the buffer point $s_b^-$ (in the tunnel region) and equal to the constant $s_c^+$ above the buffer point $s_b^-$ (in the funnel region). The push-forward probability measure $\mu_0\widetilde S_0^{-1}$ of $\mu_0$ under $\widetilde S_0$ is supported on $T\cap ]-\infty,s_c^+]$.
\begin{theorem}\label{mainthm-weak}
    Let $S_0:L\to T$ be a slow relation function associated to \eqref{system-non-generic-measure}. Let $\mu_{\tilde\epsilon},\mu_0$ be Borel probability measures supported on $L$. The following statements hold.
\begin{itemize}
    \item[(a)] If $-\tilde I_-(s_c^-)\le \tilde I_+(s_c^+)$ and if $\mu_{\tilde\epsilon}$ converges weakly to $\mu_0$ as $\tilde \epsilon\to 0$, then $\mu_{\tilde\epsilon}S_{\tilde\epsilon}^{-1}$ converges weakly to $\mu_0S_0^{-1}$ as $\tilde \epsilon\to 0$.
    \item[(b)] If $-\tilde I_-(s_c^-)>\tilde I_+(s_c^+)$ and if $\mu_{\tilde\epsilon}$ converges weakly to $\mu_0$ as $\tilde \epsilon\to 0$, then $\mu_{\tilde\epsilon}S_{\tilde\epsilon}^{-1}$ converges weakly to $\mu_0\widetilde S_0^{-1}$, as $\tilde \epsilon\to 0$. 
\end{itemize}
\end{theorem}

Using \eqref{mixture-measure-push-forward} it can be easily seen that $\mu_0\widetilde S_0^{-1}$ from Theorem \ref{mainthm-weak}(b) can be written as 
    \begin{equation}\label{mixture-measure}
    \mu_0\widetilde S_0^{-1}(\cdot)=\mu_0S_0^{-1}\left(\cdot\cap T_b\right)+\mu_0\left([s_b^-,+\infty[\right)\delta_{s_c^+}(\cdot),
    \end{equation} where $T_b:=T\cap ]-\infty,s_c^+[$.
 The first term in \eqref{mixture-measure} comes from the tunnel behaviour and the second from the funnel behaviour (see Section \ref{proof-thmweak} and Proposition \ref{prop:slowrel}(b)). If $\mu_0$ is supported on $L\cap ]-\infty,s_b^-[$ (below the buffer point $s_b^-$, in the tunnel region), then the measure in \eqref{mixture-measure} is equal to $\mu_0S_0^{-1}$, similarly to Theorem \ref{mainthm-weak}(a) where we also have the tunnel behaviour. If $\mu_0$ is supported on $L\cap [s_b^-,+\infty[$ (above the buffer point $s_b^-$, in the funnel region), then \eqref{mixture-measure} is a Dirac delta measure $\delta_{s_c^+}$. 

Theorem \ref{mainthm-weak} will be proved in Section \ref{proof-thmweak}. We know that weak convergence is preserved by continuous mappings (see Section \ref{sec-ergodic}). This property cannot be used directly because mappings $S_{\tilde\epsilon}$ depend on the singular parameter $\tilde\epsilon$. To prove Theorem \ref{mainthm-weak}(a) (resp. Theorem \ref{mainthm-weak}(b)), we will need uniform convergence of $S_{\tilde\epsilon}$ to $S_{0}$ (resp. to $\widetilde S_0$), as $\tilde \epsilon\to 0$. For more details we refer to Section \ref{proof-thmweak}.

\section{Numerical results}\label{section-numerics}

In this section, we present two numerical examples that illustrate the results presented in Section \ref{subsection-densities}. \rev{These numerical simulations are performed in Mathematica \cite{Mathematica} which, by default, uses the LSODA integration method \cite{petzold1983automatic}. We recall that the numerical integration of singularly perturbed problems is highly delicate \cite{hairer1996solving}, and in some cases, discretizations may even change the behavior of canards \cite{engel2022discretized,engel2020extended}. That is why, regarding the numerical integration, we use for all simulations a \texttt{MaxStepSize} of $\frac{1}{100}$ and a \texttt{PrecisionGoal}\footnote{That is, the number of effective digits of precision for the numerical computations.} of $50$, which we found to be enough for the numerical result to be in accordance to the theory presented above. Furthermore, we emphasize that although the initial conditions are randomly generated (via a random distribution; see more details below), the plots we show below are representative of at least $10$ distinct simulation runs. Regarding the histograms, the bin sizes are automatically set to show $10$ bins. Any further detail is mentioned when relevant.}  

The first example concerns the van der Pol equation, and we show the entry-exit behaviour for the cases $-\tilde I_-(s_c^-)\le \tilde I_+(s_c^+)$ and $-\tilde I_-(s_c^-)> \tilde I_+(s_c^+)$. In particular, we compute numerically the exit density (for $\epsilon=0$ via \eqref{Frobenius-slow-fast-Lienard}, and for $\epsilon>0$ small from numerical integration) provided that the entry density is from a uniform distribution, and compare the effect of lowering $\epsilon$. See Example \ref{example-num-Pol} below.

The second example deals with a non-generic Li\'enard equation \eqref{system-non-generic} (or equivalently \eqref{system-non-generic-measure}), and shows the entry-exit relation, and densities, for a truncated Cauchy entry density (see Example \ref{example-num-non-gen}).

\begin{example}\label{example-num-Pol}
Consider the van der Pol equation \eqref{example-classical-Lienard}. We present below numerical simulation showing the relationship between entry and exit densities of uniformly distributed initial conditions. We present the simulations for two values of the singular parameter $\epsilon$ showcasing the behaviour as $\epsilon\to0$.

\begin{description}
    \item[a)  $-\tilde I_-(s_c^-)\le \tilde I_+(s_c^+)$] for this case we choose $s_c^-=\frac{1}{20}$ and $s_c^+=\frac{1}{10}$ with $\epsilon=\frac{1}{100}$ and $\epsilon=\frac{1}{200}$ giving a corresponding value of the parameter $\tilde\lambda\approx-\frac{231}{20000}$ and $\tilde\lambda\approx-\frac{107135}{20000000}$.  This parameter gives the red orbit that connects $s_c^-$ with $s_c^+$ via an orbit for the particular chosen value of $\epsilon$, see the phase-portraits of figures \ref{fig:vdp_1_1} and \ref{fig:vdp_1_2}. For both sets of simulations, some initial conditions are chosen uniformly along the section $\sigma^-$, parametrized by the $y-$coordinate and within the interval $s\in[s_1^--\delta,s_c^-]$ with $s_1^-=\frac{1}{30}$ and $\delta=\frac{1}{150}$. The corresponding orbits are numerically computed until they arrive to the exit section $\sigma^+$, blue orbits in the phase portrait of figures \ref{fig:vdp_1_1} and \ref{fig:vdp_1_2}.  The corresponding entry and exit densities, the latter given by \eqref{Frobenius-slow-fast-Lienard}, are numerically computed and shown in the right of side of the figures. We recall that such densities correspond to the singular case $\epsilon=0$. Alongside these densities, we numerically compute a histogram of the exit coordinates of the orbits of the phase portrait (also shown on the right of the figures). This histogram corresponds to the distribution of the orbits as they cross $\sigma^+$. By comparing figures \ref{fig:vdp_1_1} and \ref{fig:vdp_1_2}, notice that as $\epsilon$ decreases, the histogram resembles more the exit distribution $D_{ex}$ (see Theorem \ref{mainthm-weak}(a)).

\begin{figure}[htbp]\centering
    \centering
    \begin{tikzpicture}
    \node at (-2,0){
        \includegraphics[]{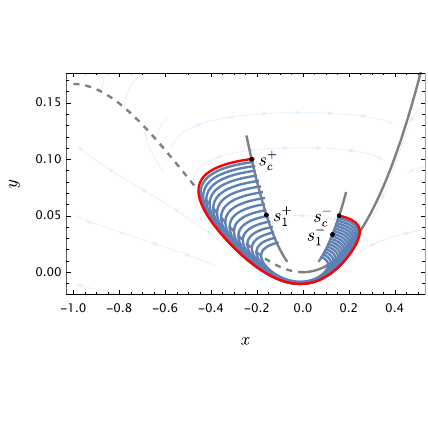}
    };
    \node at (-1.5,1.5){\footnotesize$\sigma^+$};
    \node at (0.2,0.5){\footnotesize$\sigma^-$};

    \node at (4,3){
        \includegraphics[]{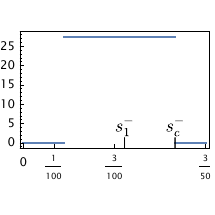}
    };
    
    \node at (4,0){
        \includegraphics[]{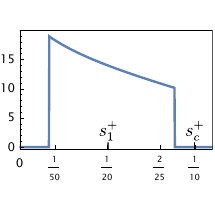}
    };
    
    \node at (4,-3){
        \includegraphics[]{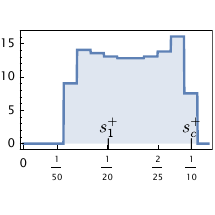}
    };
    \node at (4,4.5){\footnotesize$D_{en}(s)$};
    \node at (4,1.5){\footnotesize$D_{ex}(s)$};
    \node at (4,-1.5){\footnotesize{Exit Histogram}};
    \node at (-1.5,2.75){$\epsilon=\frac{1}{100},\;\tilde\lambda\approx-\frac{231}{20000}$};
    \end{tikzpicture}
        \caption{Numerical simulation for the case $-\tilde I_-(s_c^-)< \tilde I_+(s_c^+)$ with $\epsilon=\frac{1}{100}$. The left panel shows a phase-portrait highlighting in red the orbit for $\tilde\lambda\approx-\frac{231}{20000}$. The right panels show the entry distribution (top), exit distribution (middle), and a histogram of the exit points of the orbits crossing $\sigma^+$. The horizontal coordinate of all the right panels is the height ($y$-component) of points along the sections $\sigma^\pm$.}
    \label{fig:vdp_1_1}
\end{figure}
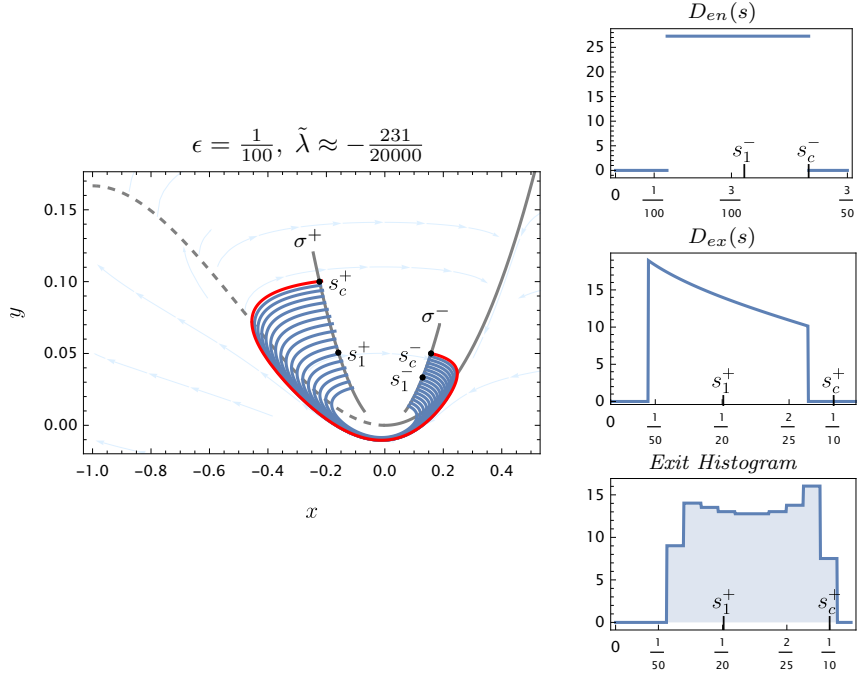
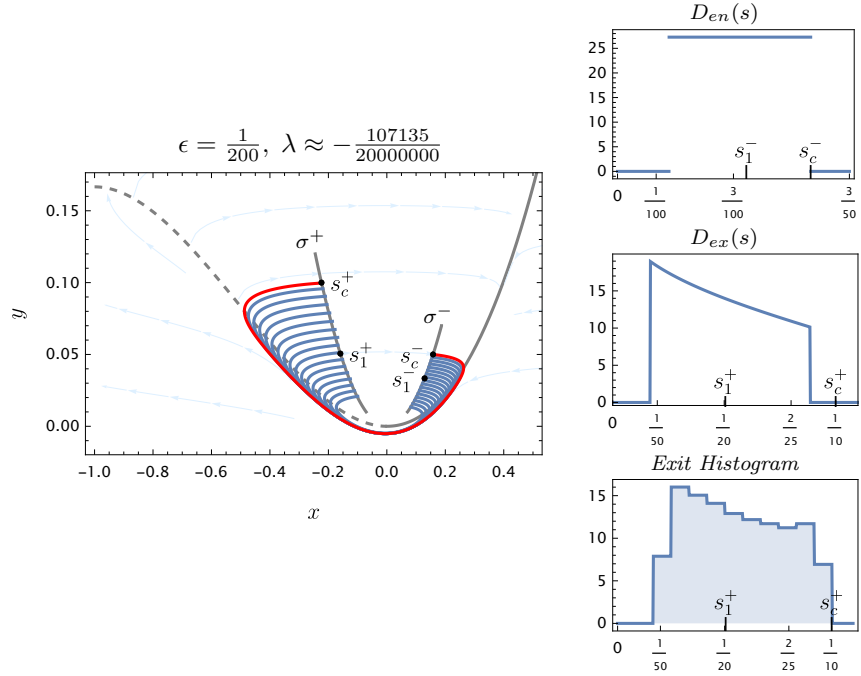
\begin{figure}[htbp]\centering
    \begin{tikzpicture}
    \node at (-2,0){
        \includegraphics[]{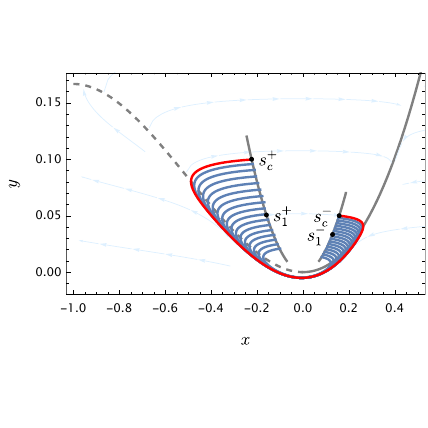}
    };
    \node at (-1.5,1.5){\footnotesize$\sigma^+$};
    \node at (0.2,0.5){\footnotesize$\sigma^-$};

    \node at (4,3){
        \includegraphics[]{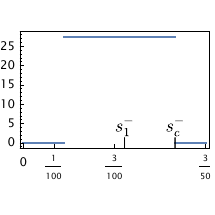}
    };
    
    \node at (4,0){
        \includegraphics[]{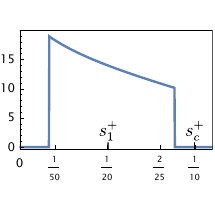}
    };
    
    \node at (4,-3){
        \includegraphics[]{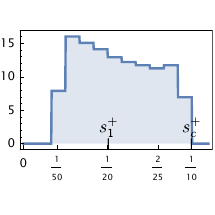}
    };
    \node at (4,4.5){\footnotesize$D_{en}(s)$};
    \node at (4,1.5){\footnotesize$D_{ex}(s)$};
    \node at (4,-1.5){\footnotesize{Exit Histogram}};
    \node at (-1.5,2.75){$\epsilon=\frac{1}{200},\;\lambda\approx-\frac{107135}{20000000}$};
    \end{tikzpicture}
    \caption{Numerical simulation for the case $-\tilde I_-(s_c^-)< \tilde I_+(s_c^+)$ with $\epsilon=\frac{1}{200}$. The left panel shows a phase-portrait highlighting in red the orbit for $\lambda\approx-\frac{107135}{20000000}$. The right panels show the entry distribution (top), exit distribution (middle), and a histogram of the exit points of the orbits crossing $\sigma^+$. The horizontal coordinate of all the right panels is the height ($y$-component) of points along the sections $\sigma^\pm$. Compare with figure \ref{fig:vdp_1_1} and notice that the exit histogram resembles more the exit density.}
    \label{fig:vdp_1_2}
\end{figure}



    \item[b)  $-\tilde I_-(s_c^-)> \tilde I_+(s_c^+)$] for this case we choose $s_c^-=\frac{1}{10}$ and $s_c^+=\frac{1}{7}$ with $\epsilon=\frac{1}{100}$ and $\epsilon=\frac{1}{200}$, as above,  giving corresponding values of the parameter $\tilde\lambda\approx-\frac{2348}{200000}$ and $\tilde\lambda\approx-\frac{1071435}{200000000}$, respectively. These parameters give the red orbits connecting $s_c^-$ with $s_c^+$, in figures \ref{fig:vdp_2_1} and \ref{fig:vdp_2_2}. For this setup, the value of $s_b^-$ (which satisfies $\tilde I_-(s_b^-)+ \tilde I_+(s_c^+)=0$) is numerically obtained as $s_b^-\approx\frac{651}{10000}$. Some initial conditions are chosen uniformly along the section $\sigma_-$, parametrized by the $y-$coordinate and within an interval around $s_b^-$, distinguishing those initial conditions with $s\leq s_b^-$ and those with $s>s_b^-$ (blue and orange orbits respectively in the phase-portraits). We notice that, as predicted by Proposition \ref{prop:slowrel}, the exit density for the orbits starting below $s_b^-$ is not ``concentrated''(tunnel behaviour), while the exit density corresponding to initial conditions above $s_b^-$ clearly look concentrated near $s_c^+$ (funnel behaviour). As in the previous example, we also compute an histogram of the coordinates of the exit points of the orbits crossing $\sigma^+$. One can indeed notice, from figures \ref{fig:vdp_2_1} and \ref{fig:vdp_2_2}, that the exit distribution corresponding to the funnel region (orbits above $s_b^-$) seems to approach to a Dirac delta as $\epsilon$ decreases, as predicted in Theorem \ref{mainthm-weak}(b).
\end{description}

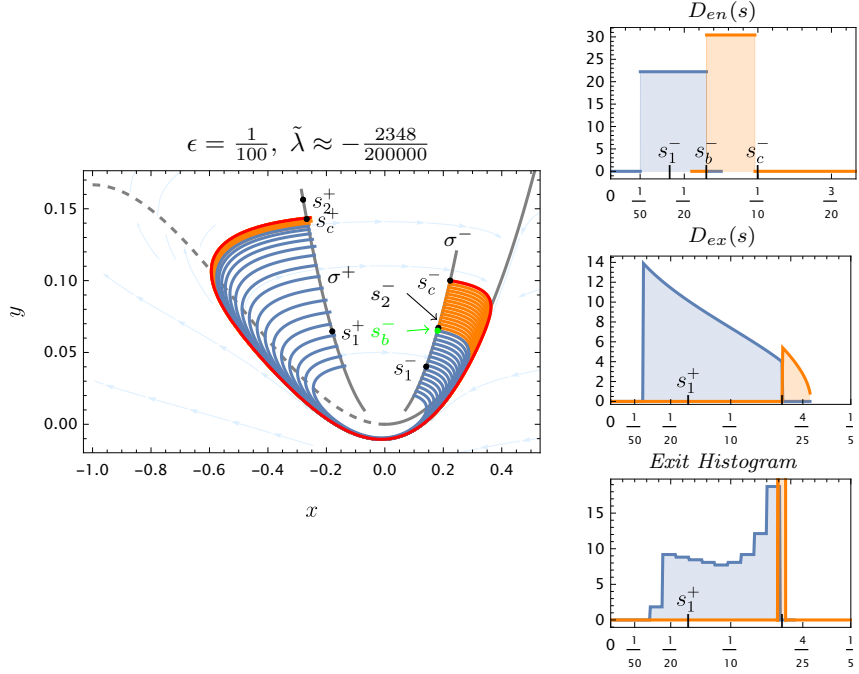
\begin{figure}[htbp]\centering
    \centering
    \begin{tikzpicture}
    \node at (-2,0){
        \includegraphics[]{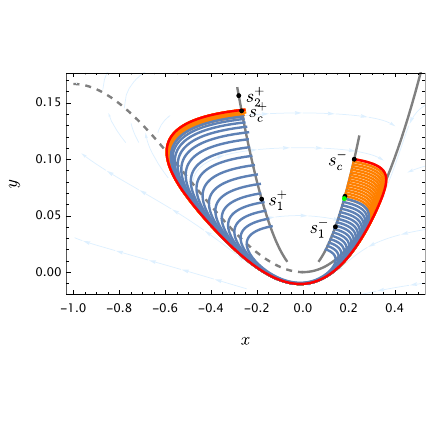}
    };
    \node at (-1.05,1){\footnotesize$\sigma^+$};
    \node at (0.5,1.5){\footnotesize$\sigma^-$};

    \node[green] (sb) at (-.5,0.25){\footnotesize$s_b^-$};
    \node (ss) at (-.5,0.75){\footnotesize$s_2^-$};

    \draw[->,green] (sb.east)--++(.3,.025);
    \draw[->,black] (ss.east)--++(.4,-.35);

    \node at (4,3){
        \includegraphics[]{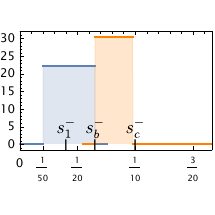}
    };
    
    \node at (4,0){
        \includegraphics[]{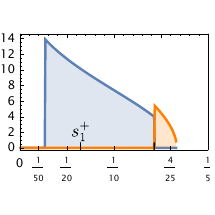}
    };
    
    \node at (4,-3){
        \includegraphics[]{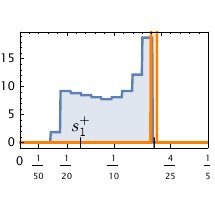}
    };
    \node at (4,4.5){\footnotesize$D_{en}(s)$};
    \node at (4,1.5){\footnotesize$D_{ex}(s)$};
    \node at (4,-1.5){\footnotesize{Exit Histogram}};
    \node at (-1.5,2.75){$\epsilon=\frac{1}{100},\;\tilde\lambda\approx-\frac{2348}{200000}$};
    \end{tikzpicture}
        \caption{Numerical simulation for the case $-\tilde I_-(s_c^-)> \tilde I_+(s_c^+)$ with $\epsilon=\frac{1}{100}$. The left panel shows a phase-portrait highlighting in red the orbit for $\tilde\lambda\approx-\frac{2348}{200000}$ that connects $s_c^-$ with $s_c^+$. The right panels show the entry distribution (top), exit distribution (middle), and a histogram of the exit points of the orbits crossing $\sigma^+$. The exit distribution is computed via \eqref{Frobenius-slow-fast-Lienard}, while the histogram corresponds to the vertical coordinates at $\sigma^+$ of $500$ orbits starting from $\sigma^-$ according to the entry density $D_{en}$.
        On the one hand, it is worth noticing that, from the histogram, the orbits that start above $s_b^-$ concentrate in $\sigma^+$ near $s_c^+$, see Proposition \ref{prop:slowrel} and Theorem \ref{mainthm-weak}(b).  On the other hand, the exit density computed via \eqref{Frobenius-slow-fast-Lienard} (in the second panel) corresponding to the funnel region (orange) is not related to the Dirac measure at $s_c^+$, recall \eqref{mixture-measure}. This same observation holds for the rest of the examples involving a funnel region, see figures \ref{fig:vdp_2_2} and \ref{fig:Ex2-c2}.}
    \label{fig:vdp_2_1}
\end{figure}
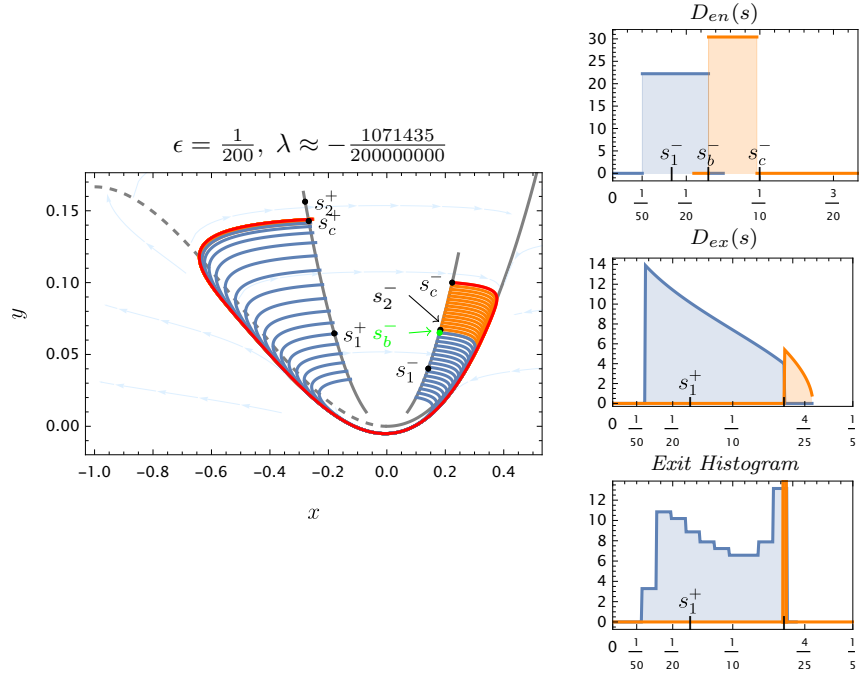
\begin{figure}[htbp]\centering
    \begin{tikzpicture}
    \node at (-2,0){
        \includegraphics[]{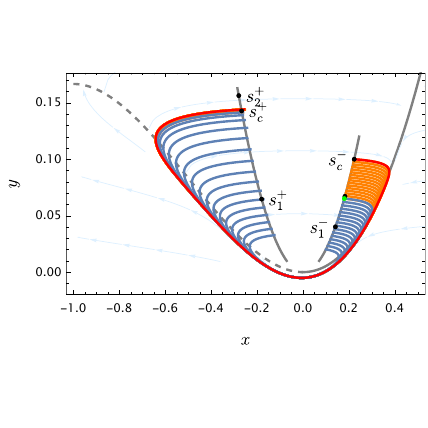}
    };
    \node[green] (sb) at (-.5,0.25){\footnotesize$s_b^-$};
    \node (ss) at (-.5,0.75){\footnotesize$s_2^-$};

    \draw[->,green] (sb.east)--++(.3,.025);
    \draw[->,black] (ss.east)--++(.4,-.35);

    \node at (4,3){
        \includegraphics[]{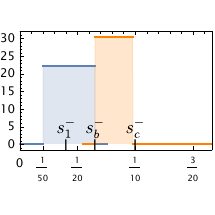}
    };
    
    \node at (4,0){
        \includegraphics[]{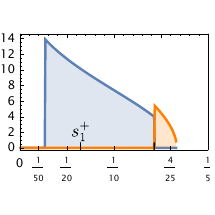}
    };
    
    \node at (4,-3){
        \includegraphics[]{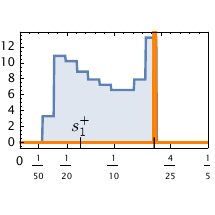}
    };
    \node at (4,4.5){\footnotesize$D_{en}(s)$};
    \node at (4,1.5){\footnotesize$D_{ex}(s)$};
    \node at (4,-1.5){\footnotesize{Exit Histogram}};
    \node at (-1.5,2.75){$\epsilon=\frac{1}{200},\;\lambda\approx-\frac{1071435}{200000000}$};
    \end{tikzpicture}
    \caption{Numerical simulation for the case $-\tilde I_-(s_c^-)> \tilde I_+(s_c^+)$, similar to the one shown in figure \ref{fig:vdp_2_1}, but with $\epsilon=\frac{1}{200}$. The left panel shows a phase-portrait highlighting in red the orbit for $\tilde\lambda\approx-\frac{1071435}{200000000}$, which connects $s_c^-$ with $s_c^+$. The right panels show the entry distribution (top), exit distribution (middle), and a histogram of the exit points of the orbits crossing $\sigma^+$. Compare with figure \ref{fig:vdp_2_1} and notice that the exit histogram resembles more the exit density for the tunnel behaviour, while for the funnel behaviour the exit histogram is thinner. This evidences the fact that according to Proposition \ref{prop:slowrel} and especially Theorem \ref{mainthm-weak}(b), the exit density in the funnel region converges to a Dirac delta. }
    \label{fig:vdp_2_2}
\end{figure}

\end{example}

\clearpage

\begin{example}\label{example-num-non-gen}
    Following a similar idea as in the previous example, let us now consider the non-generic Li\'enard equation, see \eqref{system-non-generic} (or equivalently \eqref{system-non-generic-measure})
    \begin{equation}\label{eq:ex2}
        \begin{split}
            \ddt{x}&=y- x^4\\
        \ddt{y} &= \epsilon (\lambda-x^3),
        \end{split}
    \end{equation}
    but we now (randomly) choose initial conditions $y(t_0)$ from a truncated Cauchy distribution. 
    
    A realisation for the case where $-\tilde I_-(s_c^-)\le \tilde I_+(s_c^+)$ is shown in Fig. \ref{fig:Ex2-c1}. Here $\epsilon=\frac{1}{100}$ and $\tilde\lambda=-\frac{22535}{10000000000}$. For the phase-portrait we choose $50$ initial conditions along $\sigma^-$ according to the truncated distribution ($D_{en}$) shown in the right panel of Fig. \ref{fig:Ex2-c1}. Due to the symmetry of the problem, the entry distribution, which is centred at $s_1^-$ is mapped to a distribution centred close to $s_1^+$ which has the same vertical coordinate. As $\epsilon\to0$, and due to the symmetry again, the exit density along $\sigma^+$ converges (weakly) to the entry density, which is visible in the Figure (keep in mind that the vertical coordinate of $s_1^+$ coincides with that of $s_1^-$ in the limit $\epsilon=0$). We also show a histogram of the vertical coordinates at $\sigma^+$ of $500$ trajectories with initial conditions in $\sigma^-$ according to $D_{en}$. 
    
    \begin{figure}
        \centering
         \begin{tikzpicture}
    \node at (-2,0){
        \includegraphics[]{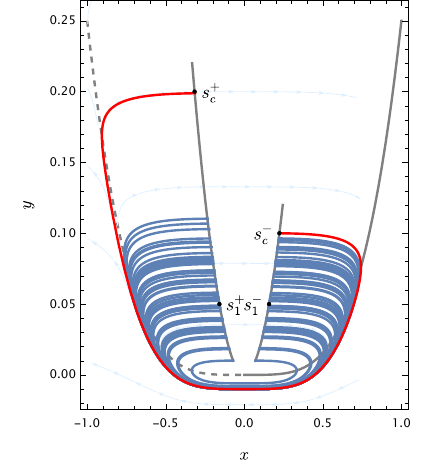}
    };
    \node at (-2.1,3){\footnotesize$\sigma^+$};
    \node at (-.6,.8){\footnotesize$\sigma^-$};

    \node at (4,3){
        \includegraphics[]{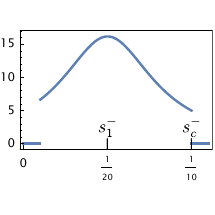}
    };
    
    \node at (4,0){
        \includegraphics[]{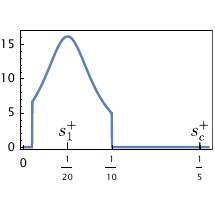}
    };
    
    \node at (4,-3){
        \includegraphics[]{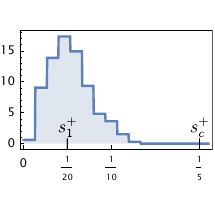}
    };
    \node at (4,4.5){\footnotesize$D_{en}(s)$};
    \node at (4,1.5){\footnotesize$D_{ex}(s)$};
    \node at (4,-1.5){\footnotesize{Exit Histogram}};
    \node at (-1.5,4.25){$\epsilon=\frac{1}{100},\;\lambda\approx-\frac{22535}{10000000000}$};
    \end{tikzpicture}
        \caption{Numerical simulation of the entry-exit tunnel behaviour ($-\tilde I_-(s_c^-)\le \tilde I_+(s_c^+)$) for \eqref{eq:ex2}. The left panel shows a phase portrait where the height of the initial conditions along $\sigma^-$ are chosen according to $D_{en}$. We also show on the right the corresponding exit density $D_{ex}$ computed with \eqref{Frobenius-slow-fast-Lienard} (we recall that this map is for $\epsilon=0$). The histogram shows the distribution of the heights along $\sigma^+$ of the numerical integration of $500$ orbits starting on $\sigma^-$ according to the entry density, and the parameters $(\epsilon,\tilde\lambda)$ previously mentioned.}
        \label{fig:Ex2-c1}
    \end{figure}
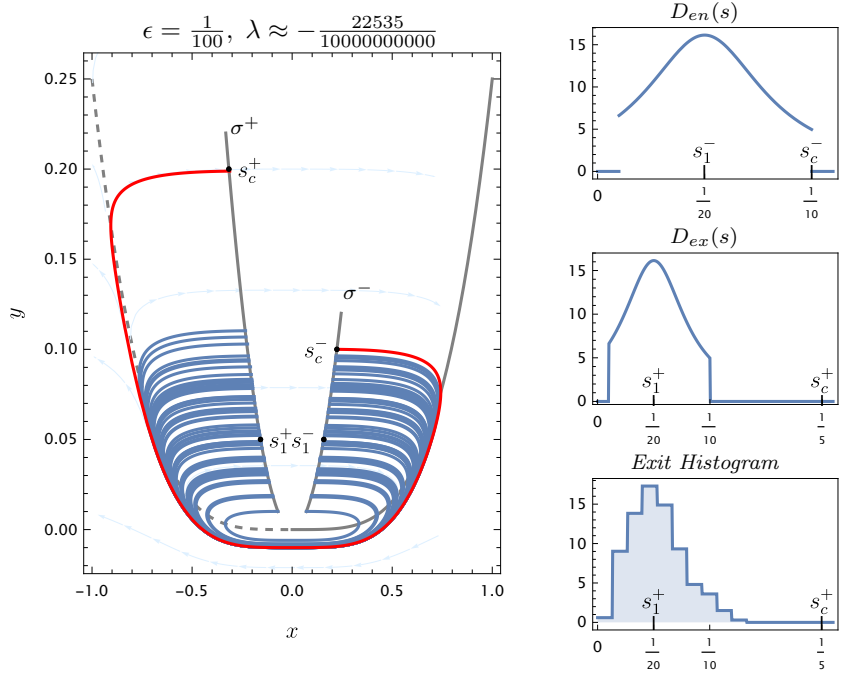

Analogously, a realisation for the case where $-\tilde I_-(s_c^-)> \tilde I_+(s_c^+)$ is shown in Fig. \ref{fig:Ex2-c2}. Here $\epsilon=\frac{1}{100}$, $\lambda=\frac{2}{1000000}$, and we also choose $50$ initial conditions along $\sigma^-$ according to the distribution ($D_{en}$) shown in the right panel of Fig. \ref{fig:Ex2-c2}. Similar to the previous example, we see a contrast between the orbits below (tunnel region) and those above (funnel region) $s_b^-$ which is translated into an equivalent exit distribution ($D_{ex}$) and corresponding exit histogram as indicated in Proposition \ref{prop:slowrel} and Theorem \ref{mainthm-weak}. In particular it is evident that the orbits that start above $s_b^-$ are concentrated at $\sigma^+$ near $s_c^+$. 
\begin{figure}
        \centering
        \begin{tikzpicture}
    \node at (-2,0){
        \includegraphics[]{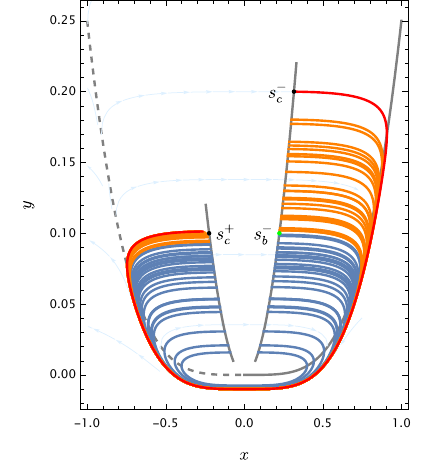}
    };
    \node at (-2.1,0.75){\footnotesize$\sigma^+$};
    \node at (-.5,3.2){\footnotesize$\sigma^-$};

    \node at (4,3){
        \includegraphics[]{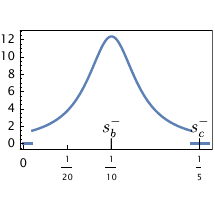}
    };
    
    \node at (4,0){
        \includegraphics[]{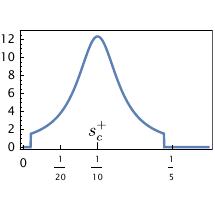}
    };
    
    \node at (4,-3){
        \includegraphics[]{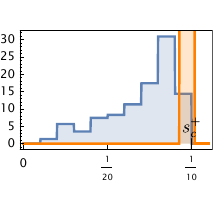}
    };
    \node at (4,4.5){\footnotesize$D_{en}(s)$};
    \node at (4,1.5){\footnotesize$D_{ex}(s)$};
    \node at (4,-1.5){\footnotesize{Exit Histogram}};
    \node at (-1.5,4.25){$\epsilon=\frac{1}{100},\;\lambda\approx\frac{2}{1000000}$};
    \end{tikzpicture}
        \caption{Numerical simulation of the entry-exit behaviour of \eqref{eq:ex2} for the case $-\tilde I_-(s_c^-)> \tilde I_+(s_c^+)$, showing tunnel (blue) and funnel (orange) behaviour. The left panel shows a phase portrait where the height of the initial conditions along $\sigma^-$ are chosen according to $D_{en}$. We also show on the right the corresponding exit density $D_{ex}$ computed with \eqref{Frobenius-slow-fast-Lienard} (we recall that this map is for $\epsilon=0$). The histogram shows the distribution of the heights along $\sigma^+$ of the numerical integration of $500$ orbits starting on $\sigma^-$ according to the entry density, and the parameters $(\epsilon,\tilde\lambda)$ previously mentioned.}
        \label{fig:Ex2-c2}
    \end{figure}
    
\end{example}

\section{Proof of Theorem \ref{theorem-1}--Theorem \ref{theorem-3}}\label{section-proofs}
In this section we prove Theorem \ref{theorem-1}, Theorem \ref{theorem-2} and Theorem \ref{theorem-3}. We assume that  Assumption 1--Assumption 4 are always satisfied. Following Definition \ref{def-1}, if $-\tilde I_-(s_0)\le \tilde I_+(s_0)$, then the slow relation function $S:[0,s_0]\to [0,s_0]$, $S(0)=0$, satisfies
  $$  \tilde I_-(s)+ \tilde I_+(S(s))=0,$$
  for $s\in ]0,s_0]$. This and \eqref{SDI-final-version} imply that for $s\in ]0,s_0]$
  \begin{align}\label{eq-proof1}
      \tilde I(s)&= \tilde I_-(s)+ \tilde I_+(s)\nonumber\\
      &= \tilde I_-(s)+\tilde I_+(S(s))+ \tilde I_+(s)-\tilde I_+(S(s))\nonumber\\
      &=\tilde I_+(s)-\tilde I_+(S(s)).
  \end{align}
Let us recall that $\tilde I_+'(s)>0$ for $s\in ]0,s_0]$ (Section \ref{subsection-model}). From this property and \eqref{eq-proof1} it follows that $s_1\in ]0,s_0]$ is a zero of $\tilde I$ if and only if $s_1$ is a fixed point of the slow relation function $S$. Moreover, if we define a smooth positive function $\Psi(s,w)$ for $s,w\in ]0,s_0]$: $\Psi(s,w)=\frac{\tilde I_+(s)-\tilde I_+(w)}{s-w}$ if $s\ne w$ and $\Psi(s,s)=\tilde I_+'(s)$, then using \eqref{eq-proof1} we get
$$ \tilde I(s)=\Psi(s,S(s))\left(s-S(s)\right),$$
for $s\in ]0,s_0]$. We conclude that $s_1\in ]0,s_0]$ is a zero of $\tilde I$ of multiplicity $l$ if and only if $s_1$ is a zero of $s-S(s)$ of multiplicity $l$. 

If $-\tilde I_-(s_0)> \tilde I_+(s_0)$, then the slow relation function $S:[0,s_0]\to [0,s_0]$, $S(0)=0$, satisfies
  $\tilde I_-(S(s))+ \tilde I_+(s)=0$
  for $s\in ]0,s_0]$, and the study of this case is analogous to the study of the case where $-\tilde I_-(s_0)\le \tilde I_+(s_0)$.

\subsection{Proof of Theorem \ref{theorem-1}}\label{section-proof-1}
We will use the following topological version of the Poincar\'{e} recurrence theorem (see \cite[Theorem 1.2.4]{viana2016foundations}).
\begin{theorem}\label{thm-Poin}
    Let $X$ be a topological space, endowed with its Borel
$\sigma$-algebra $\mathcal B$. Assume that $X$ admits a countable
basis of open sets and that $f:X\to X$ is a measurable transformation. If $\mu$ is an $f$-invariant probability measure on $X$, then $\mu$-almost every $x\in X$ is recurrent for $f$.
\end{theorem}
We say that $x\in X$ is recurrent for $f:X\to X$ if $f^{n_i}(x)\to x$ for some sequence $n_i\to\infty$. Whenever we say that some property holds for $\mu$-almost every $x\in X$ we mean that the said property holds for all $x\in X\backslash Y$, with $\mu(Y)=0$.

If $X$ is the compact metric space $[0,s_0]$ and $f$ is the slow relation function $S:[0,s_0]\to [0,s_0]$ (recall that $S$ is continuous), then assumptions of Theorem \ref{thm-Poin} are satisfied.\\
\\
\textit{Proof of Theorem \ref{theorem-1}.1.} Since $S(0)=0$, $\delta_0$ is $S$-invariant. We know that $\tilde I$ has no zeros in $]0,s_0]$ if and only if $S$ has no fixed points in $]0,s_0]$ (see the paragraph after \eqref{eq-proof1}). 

Since $S$ is increasing on $[0,s_0]$, for each $s\in [0,s_0]$ the sequence $S^n(s)$ is bounded and monotone (thus, convergent) and its limit has to be a fixed point of $S$. This implies that $s\in [0,s_0]$ is recurrent for $S$ if and only if $s$ is a fixed point of $S$. 

Assume that $S$ has no fixed points in $]0,s_0]$ ($s=0$ is the unique recurrent point). Then Theorem \ref{thm-Poin} implies that for each $S$-invariant probability measure $\mu$ on $[0,s_0]$ we have $\mu(\{0\})=1$. We conclude that $\mu=\delta_0$ and $S$ is therefore uniquely ergodic. 

Suppose now that $S$ is uniquely ergodic. Then there is a unique $S$-invariant probability
measure ($\delta_0$). It is clear that $S$ has no fixed points in $]0,s_0]$ (if $S(s)=s$ for some $s\in ]0,s_0]$, then $\delta_s$ is a new $S$-invariant probability
measure). This completes the proof of Theorem \ref{theorem-1}.1.\\
\\
\textit{Proof of Theorem \ref{theorem-1}.2.} We know that $s_1\in ]0,s_0]$ is a zero of $\tilde I$ if and only if $s_1$ is a fixed point of $S$. Now, it suffices to notice that a Dirac measure $\delta_{s_1}$ is
$S$-invariant if and only if $s_1$ is a fixed point of $S$.\\
\\
\textit{Proof of Theorem \ref{theorem-1}.3.} Suppose that $\tilde I$ has $k$ zeros $s_1<\dots <s_k$ in $]0,s_0]$. Then $S$ has $k+1$ fixed points $0,s_1,\dots,s_k$ in $[0,s_0]$ and $\delta_0,\delta_{s_1},\dots,\delta_{s_k}$ are $S$-invariant. Thus, $0,s_1,\dots,s_k$ are the unique recurrent points for $S$ (see the proof of Theorem \ref{theorem-1}.1) and Theorem \ref{thm-Poin} implies that for every $S$-invariant probability measure $\mu$ on $[0,s_0]$ we have $\mu(\{0,s_1,\dots,s_k\})=1$ and 
$$\mu=\mu(\{0\})\delta_0+\mu(\{s_1\})\delta_{s_1}+\dots+\mu(\{s_k\})\delta_{s_k}.$$
Since the set of all $S$-invariant probability measures is convex, we get \eqref{convex-set}.

Conversely, suppose that the set $\mathcal P_S$ of all $S$-invariant probability measures is given by \eqref{convex-set}. Then $\delta_s\in \mathcal P_S$ if and only if $s\in\{0,s_1,\dots,s_k\}$. Then $S$ has $k$ fixed points $s_1,\dots,s_k$ in $]0,s_0]$. Thus, $\tilde I$ has $k$ zeros in $]0,s_0]$. This completes the proof of Theorem \ref{theorem-1}.3.

\subsection{Proof of Theorem \ref{theorem-2}}\label{section-proof-2}
We suppose that $X_{\lambda,\epsilon}$ has a slow-fast Hopf point at $p_0$ for $\lambda=\lambda_0$. Let $S:[0,s_0]\to [0,s_0]$ be the slow relation function. \\
\\
\textit{Proof of Theorem \ref{theorem-2}.1.} Assume that $S$ is uniquely ergodic. Then Theorem \ref{theorem-1}.1 implies that the slow divergence integral $\tilde I$ has no zeros in $]0,s_0]$. Following \cite[Proposition 2.2]{DM} or \cite{de2021canard}, we have $\cycl(X_{\lambda,\epsilon},\Gamma_s)\le 1$ for all $s\in]0,s_0]$, and the limit cycle, if it exists, is hyperbolic and attracting (resp. repelling) if $\tilde I(s)<0$ (resp. $>0$).\\
\\
\textit{Proof of Theorem \ref{theorem-2}.2.}
Suppose that a Dirac delta measure $\delta_{s_1}$ is $S$-invariant for some $s_1\in ]0,s_0]$. Then from Theorem \ref{theorem-1}.2 it follows that $\tilde I$ has a zero at $s=s_1$ with the multiplicity equal to the multiplicity of the fixed point $s_1$ of $S$, denoted by $l$ (see also the paragraph after \eqref{eq-proof1}). If $l<\infty$, then \cite[Proposition 2.3]{DM} (or \cite{de2021canard}) implies that $\cycl(X_{\lambda,\epsilon},\Gamma_{s_1})\le l+1$.

\subsection{Proof of Theorem \ref{theorem-3}}\label{section-proof-3}
We focus on \eqref{system-non-generic} with a fixed $\mathbbmss n_1\ge 1$ and assume that the set of all $S$-invariant probability measures is given by \eqref{convex-set} for some $0<s_1<\dots <s_k<s_0$. Then Theorem \ref{theorem-1}.3 implies that $s_1,\dots,s_k$ are zeros of $\tilde I$ in $]0,s_0]$. Thus, if we take any $s_{k+1}\in ]s_k,s_0]$, then $\tilde I(s_{k+1})\ne 0$. Since we assume that the fixed points $s_1,\dots,s_k$ of $S$ are hyperbolic, we have that $s_1,\dots,s_k$ are simple zeros of $\tilde I$. Now, Theorem \ref{theorem-3} follows from \cite[Theorem 2]{SDICLE1} (see also \cite{DM-entryexit}).

\section{Proof of Theorem \ref{mainthm-weak}}\label{proof-thmweak}

\textit{Proof of Theorem \ref{mainthm-weak}(a).} Assume that $-\tilde I_-(s_c^-)\le \tilde I_+(s_c^+)$ and that $\mu_{\tilde\epsilon}$ converges  weakly to $\mu_0$, i.e.
\begin{equation}\label{eq-weak-1}
\lim_{\tilde\epsilon\to 0}\int_L\chi(s^-)\mu_{\tilde\epsilon}(ds^-)=\int_L\chi(s^-)\mu_{0}(ds^-),
\end{equation}
for every bounded, continuous function $\chi:\mathbb R\to\mathbb R$ ($\mu_{\tilde\epsilon},\mu_0$ are supported on $L$ and we may use $L$ instead of $\mathbb R$ in the definition of weak convergence, see Section \ref{sec-ergodic}). For a bounded and continuous function $\chi:\mathbb R\to\mathbb R$ we have
\begin{align}\label{eq-weak-2}  \int_{S_{\tilde\epsilon}(L)}\chi(s^+)\mu_{\tilde\epsilon}S_{\tilde\epsilon}^{-1}(ds^+)&=\int_{L}\chi(S_{\tilde\epsilon}(s^-))\mu_{\tilde\epsilon}(ds^-)\nonumber\\
    &=\int_{L}\left(\chi(S_{\tilde\epsilon}(s^-))-\chi(S_{0}(s^-))\right)\mu_{\tilde\epsilon}(ds^-)\nonumber\\
    & \ \ \ +\int_{L}\chi(S_{0}(s^-))\mu_{\tilde\epsilon}(ds^-),
\end{align}
where in the first step we use a well-known formula for the integration under a push-forward measure (see e.g. \cite[Section 2]{Bill}). Since $\chi\circ S_{0}$ is bounded and continuous, from \eqref{eq-weak-1} it follows that the second integral in \eqref{eq-weak-2} converges to $\int_{L}\chi(S_{0}(s^-))\mu_{0}(ds^-)=\int_{T}\chi(s^+)\mu_{0}S_{0}^{-1}(ds^+)$ as $\tilde \epsilon\to 0$ (again we use the above mentioned formula for integration). Thus, it suffices to show that the first integral in \eqref{eq-weak-2} converges to $0$ as $\tilde \epsilon\to 0$. Then we have that $\mu_{\tilde\epsilon}S_{\tilde\epsilon}^{-1}$ converges weakly to $\mu_0S_0^{-1}$.

It is clear that there exists a bounded segment $\widetilde T$ (for example, $\widetilde T=[0,s_c^+]$) such that $S_{\tilde\epsilon}(L)\subset\widetilde T$ for all $\tilde \epsilon\in[0,\tilde\epsilon_0]$, with a sufficiently small $\tilde\epsilon_0>0$. Let $\varrho_1>0$ be an arbitrary and fixed real number. Since $\chi$ is uniformly continuous on $\widetilde T$, there exists a $\varrho_2>0$ such that for every $x,y\in \widetilde T$ with $|x-y|<\varrho_2$ we have $$|\chi(x)-\chi(y)|<\varrho_1.$$
Since $S_{\tilde\epsilon}$ converges to $S_{0}$ as $\tilde\epsilon\to 0$, uniformly in $L$ (see Proposition \ref{prop:slowrel}(a)), for all $\tilde \epsilon\in]0,\tilde\epsilon_0]$ and $s^-\in L$ we have 
$$|S_{\tilde\epsilon}(s^-)-S_0(s^-)|<\varrho_2,$$
up to shrinking $\tilde\epsilon_0$ if needed. Putting all this together, for $\tilde \epsilon\in]0,\tilde\epsilon_0]$ we get 
\begin{align}\label{eq-weak-3}  \left|\int_{L}\left(\chi(S_{\tilde\epsilon}(s^-))-\chi(S_{0}(s^-))\right)\mu_{\tilde\epsilon}(ds^-)\right| &\le \int_{L}\left|\chi(S_{\tilde\epsilon}(s^-))-\chi(S_{0}(s^-))\right|\mu_{\tilde\epsilon}(ds^-) \nonumber\\
    & \ \ \ <\int_L\varrho_1\mu_{\tilde\epsilon}(ds^-)=\varrho_1,\nonumber
\end{align}
where in the last step we use the fact that $\mu_{\tilde\epsilon}$ is a probability measure supported on $L$. Thus, we have proved that for every $\varrho_1>0$ there is $\tilde\epsilon_0>0$ (small enough) such that the above inequality holds for all $\tilde \epsilon\in]0,\tilde\epsilon_0]$. This implies that the first integral in \eqref{eq-weak-2} converges to $0$ as $\tilde \epsilon\to 0$.
This completes the proof of Theorem \ref{mainthm-weak}(a). \\
\\
\textit{Proof of Theorem \ref{mainthm-weak}(b).} Suppose that $-\tilde I_-(s_c^-)>\tilde I_+(s_c^+)$ and that $\mu_{\tilde\epsilon}$ converges weakly to $\mu_0$ as $\tilde \epsilon\to 0$, see \eqref{eq-weak-1}. Let us recall that the function $\widetilde S_0:L\to T\cap ]-\infty,s_c^+]$ is defined in \eqref{mixture-measure-push-forward}. It suffices to show that $S_{\tilde\epsilon}$ converges to $\widetilde S_{0}$ as $\tilde\epsilon\to 0$, uniformly in $L$. Then the proof of (b) is analogous to the proof of (a) (we replace $S_0$ with $\widetilde S_{0}$ and the segment $T$ with the segment $T\cap ]-\infty,s_c^+]$).

Let us prove that $S_{\tilde\epsilon}$ uniformly converges to $\widetilde S_{0}$ as $\tilde\epsilon\to 0$. Let $\tilde\varrho_1>0$ be an arbitrarily small but fixed real number. Using Proposition \ref{prop:slowrel}(b) (the tunnel region) we may assume that $s_c^+-\frac{\tilde\varrho_1}{2}\in S_{\tilde \epsilon}(L)$ for all $\tilde \epsilon> 0$ small enough.

Since $\widetilde S_{0}$ is continuous in the buffer point $s_b^-$ ($s_b^-$ is in the interior of $L$) and $\widetilde S_0(s_b^-)=s_c^+$, there is a $\tilde\varrho_2>0$ small enough such that for every $s^-\in L$ with $|s^--s_b^-|<\tilde\varrho_2$ we have 
\begin{equation}
\label{proofb2tilde}
|\widetilde S_{0}(s^-)-s_c^+|<\frac{\tilde\varrho_1}{2}.
\end{equation} 

Proposition \ref{prop:slowrel} implies that $S_{\tilde\epsilon}^{-1}(s_c^+-\frac{\tilde\varrho_1}{2})\to S_{0}^{-1}(s_c^+-\frac{\tilde\varrho_1}{2})$ as $\tilde \epsilon\to 0$ and $S_{0}^{-1}(s_c^+-\frac{\tilde\varrho_1}{2})<s_b^-$. (Indeed, first we apply $(x,t)\to (-x,-t)$ to \eqref{system-non-generic-measure}, with $\tilde \lambda=\tilde\lambda_c(\tilde \epsilon)$. The new system is of type \eqref{system-non-generic-measure}, with $\tilde \lambda=-\tilde\lambda_c(\tilde \epsilon)$, having the orbit connecting $s_c^+$ with $s_c^-$, and having $S_{0}^{-1}$ as the slow relation function. Then it suffices to apply Proposition \ref{prop:slowrel}(a) to the new system.) From this property it follows that $S_{\tilde \epsilon}^{-1}(s_c^+-\frac{\tilde\varrho_1}{2})<s_b^--\tilde\varrho_2<s_b^-$ for every $\tilde \epsilon\in]0,\tilde\epsilon_0]$, with $\tilde\epsilon_0>0$ small enough (we take a smaller $\tilde\varrho_2>0$ if necessary and fix it).   Then, since system \eqref{system-non-generic-measure}, with $\tilde \lambda=\tilde\lambda_c(\tilde \epsilon)$, has the orbit connecting $s_c^-\in \sigma_-$ with $s_c^+\in \sigma_+$ and the segment $L$ lies below $s_c^-$ (see Fig. \ref{fig-entry-exit}(b)), we get 
\begin{equation}\label{formulatilda}
s_c^+-\frac{\tilde\varrho_1}{2}<S_{\tilde\epsilon}(s^-)<s_c^+,
\end{equation}
for all $s^-\in L\cap]s_b^--\tilde\varrho_2,+\infty[$ and $\tilde \epsilon\in ]0,\tilde\epsilon_0]$.

Now, we have 
\begin{align}\label{formulatilda+}
 |S_{\tilde\epsilon}(s^-)-\widetilde S_0(s^-)|\le  |S_{\tilde\epsilon}(s^-)-s_c^+|+|s_c^+-\widetilde S_0(s^-)|< \frac{\tilde\varrho_1}{2}+\frac{\tilde\varrho_1}{2}=\tilde\varrho_1,
\end{align}
for all $s^-\in L\cap]s_b^--\tilde\varrho_2,+\infty[$ and $\tilde \epsilon\in ]0,\tilde\epsilon_0]$. We used \eqref{proofb2tilde}, \eqref{formulatilda} and the fact that $\widetilde S_0(s^-)=s_c^+$ for $s^-\in L\cap[s_b^-,+\infty[$, see \eqref{mixture-measure-push-forward}.

On the other hand, since $S_{\tilde\epsilon}$ converges to the slow relation function $S_{0}$ as $\tilde\epsilon\to 0$, uniformly in the compact set $L\cap]-\infty,s_b^--\tilde\varrho_2]$ (see the tunnel case in Proposition \ref{prop:slowrel}(b)) and $\widetilde S_0(s^-)=S_0(s^-)$ for $s^-\in L\cap]-\infty,s_b^--\tilde\varrho_2]$, we get
\begin{align}\label{formulatilda-}
 |S_{\tilde\epsilon}(s^-)-\widetilde S_0(s^-)|<\tilde\varrho_1,
\end{align}
for all $s^-\in L\cap]-\infty,s_b^--\tilde\varrho_2]$ and $\tilde \epsilon\in ]0,\tilde\epsilon_0]$ (up to shrinking $\tilde\epsilon_0$ if necessary).

Combining \eqref{formulatilda+} and \eqref{formulatilda-} we obtain the uniform convergence on $L$. This completes the proof of Theorem \ref{mainthm-weak}(b).

\bibliographystyle{abbrv}
\bibliography{bibliography}

\end{document}